\documentclass{amsart}
\newtheorem{theorem}{Theorem}[section]
\newtheorem{corollary}[theorem]{Corollary}
\newtheorem{lemma}[theorem]{Lemma}

\newtheorem{remark}[theorem]{Remark}

\usepackage[utf8]{inputenc}

\usepackage{amssymb}
\usepackage{amsmath}
\usepackage{amsthm}
\usepackage{mathtools}
\usepackage{multirow}
\usepackage{thmtools}

\usepackage{tikz}
\usetikzlibrary{calc}
\usetikzlibrary{braids}
\usepackage[colorlinks=true, citecolor=blue, backref=page]{hyperref}

\DeclareMathOperator*{\volume}{\operatorname{vol}}

\begin{document}
	
\title[Upper bounds for volumes of generalized hyperbolic polyhedra]{Upper bounds for volumes of generalized hyperbolic polyhedra and hyperbolic links}

\author[Andrey Egorov]{Andrey Egorov}
\address{Sobolev Institute of Mathematics, 630090 Novosibirsk, Russia; 
Novosibirsk State University, 630090 Novosibirsk, Russia}
\email{a.egorov2@g.nsu.ru}

\author[Andrei Vesnin]{Andrei Vesnin}
\address{Sobolev Institute of Mathematics, 630090 Novosibirsk, Russia; 
Novosibirsk State University, 630090 Novosibirsk, Russia}
\email{vesnin@math.nsc.ru}

\thanks{The authors were supported by the Theoretical Physics and Mathematics Advancement Foundation "BASIS". A.V. was also supported by the state contract of the Sobolev Institute of Mathematics (project no. FWNF-2022-0004).}
	
\begin{abstract} 
 A polyhedron in a three-dimensional hyperbolic space is said to be generalized if finite, ideal and truncated vertices are admitted. In virtue of  Belletti's theorem (2021) the exact upper bound for volumes of generalized hyperbolic polyhedra with the same one-dimensional skeleton $G$ is equal to the volume of an ideal right-angled hyperbolic polyhedron whose one-dimensional skeleton is the medial graph for $G$.  In the present paper we give the upper bounds for the volume of an arbitrary generalized hyperbolic polyhedron, where the bonds linearly depend on the number of edges. Moreover, it is shown that the bounds can be improved if the polyhedron has triangular faces and  trivalent vertices. As an application there are obtained new upper bounds for the volume of the complement to the hyperbolic link having more than eight twists in a diagram.
 \end{abstract} 
	
\keywords{hyperbolic space, volumes of hyperbolic polyhedra, hyperbolic knots and links, augmented links} 
\subjclass[2000]{52B10, 51M10, 57M25}	
	
\maketitle
	

\section{Introduction}
	
We consider convex polyhedra of finite volume in the Lobachevsky space (hyperbolic space) $\mathbb H^3$. A polyhedron in a space of constant section curvature, $\mathbb S^n$, $\mathbb E^n$ or $\mathbb H^n$, is said to be \emph{acute-angled}, see~\cite{An70a}, or \emph{free from obtuse dihedral angles}, see~\cite{Co34}, if all its dihedral angles do not exceed $\pi/2$. In particular, a polyhedron is said to be \textit{right-angled} if all its dihedral angles are equal to $\pi/2$. It is well known that in the spherical space $\mathbb S^n$ any acute-angled polyhedron is a simplex~\cite[Theorem~1]{Co34}, and in the Euclidean space $\mathbb E^n$ any acute-angled polyhedron is a simplicial prism~\cite[Theorem~2]{Co34}.

The necessary and sufficient conditions for the realization of a polyhedron of a given combinatorial type with prescribed dihedral angles as an acute-angled polyhedron in $\mathbb H^3$ of finite volume are described by Andreev~\cite{An70a, An70b}, see also \cite{RHD07}. These conditions are formulated in the form of linear equations and inequalities, which are determined by the combinatorics of the one-dimensional skeleton (1-skeleton) of a polyhedron. Moreover, if a realization of a polyhedron in $\mathbb H^3$ exists, then it is unique up to the isometry of the space. Thus, the volume of an acute-angled hyperbolic polyhedron is completely determined by the combinatorics of its 1-skeleton and by dihedral angles. Denote $\overline{\mathbb H^3} = \mathbb H^3 \cup \partial \mathbb H^3$. The vertex of a hyperbolic polyhedron is said to be \textit{ideal} if it belongs to the absolute $\partial\mathbb H^3$. A polyhedron will be called \textit{ideal} if all its vertices are ideal. It follows from~\cite{An70b} that if $P$ is an ideal right-angled hyperbolic polyhedron, then each of its vertices is 4-valent, i.e. incident to exactly four edges. 

Calculation of the volume of a hyperbolic polyhedron given by its combinatorics and dihedral angles is a rather difficult problem. A solution of this problem for a particular family of tetrahedra goes back to Lobachevsky. Some modern results and methods related to the problem are presented in works of Milnor~\cite{Mi82}, Kellerhals~\cite{Ke89}, Vinberg~\cite{Vi93}, Kashaev~\cite{Ka97}, Cho and Kim~\cite{CK99}, Murakami and Yano~\cite{MY05}, where polyhedra with finite, ideal, or truncated vertices were under considerations. Moreover, for some classes of hyperbolic polyhedra of fixed combinatorics, such as simplexes and pyramids, there are known volumes bounds depending of number of vertices or edges. Due to the Mostow rigidity theorem, calculations of volumes and volume bounds have strigthforward  applications in the theory of hyperbolic 3-manifolds and in the knot theory~\cite{Th80}.  

Below in the formulae for the volumes of three-dimensional hyperbolic polyhedra and manifolds we will use the \emph{Lobachevsky function} introduced by Milnor in~\cite{Mi82}, 
$$
\Lambda(\theta) = - \int\limits_0^{\theta} \log | 2 \sin (t) | \, {\rm d} t.
$$
To formulate results on upper and lower volume bounds the two constants will be used which have the following values with an accuracy of up to six digits: 
$$
v_{tet} = 3\Lambda (\pi/3)= 1.014941 \quad \text{and} \qquad v_{oct} = 8\Lambda (\pi/4) = 3.663863.
$$ 
Approximate numerical values of quantities expressed in terms of the Lobachevsky function will be given with the same accuracy up to six digits. 

In the preset paper we will give the upper bounds for the volume of generalized hyperbolic polyhedra, where the bounds linearly depend on the number of edges. In Section~\ref{sec2} we recall the definition of a generalized hyperbolic polyhedron. It was shown by Belletti in~\cite{Be21} that the maximum volume of generalized hyperbolic polyhedra with the same 1-skeleton is achieved on the corresponding ideal right-angled hyperbolic polyhedron, see Theorem~\ref{theoremBeletti}. Bounds for the volumes of ideal right-angled hyperbolic polyhedra in terms of the number of vertices were previously obtained in~\cite{ABEV, At09, EV20-1, EV20-2}. Basing on these results, in the Theorem~\ref{theoremPolyhedraBounds} we obtain the upper bounds for the volumes of generalized hyperbolic polyhedra given as a linear function of the number of edges. 

\smallskip 
\noindent 
\textbf{Theorem~2.2.} \emph{Let $\Gamma$ be a 3-connected planar graph with $E$ edges, and $P$ be a generalized hyperbolic polyhedron for which $\Gamma$ is a 1-skeleton. Then the following inequalities hold. 
	\begin{itemize}
		\item[(a)] If $P$ is a tetrahedron, then $\volume(P) \leq v_{oct}$.
		\item[(b)] If $P$ is not a tetrahedron, then
		$$
		\volume (P) \leq \frac{v_{oct}}{2} \cdot E - \frac{5 v_{oct}}{2}.
		$$
		\item[(c)] If  $E > 24$, then
		$$
		\volume (P) \leq \frac{v_{oct}}{2} \cdot E - 3 v_{oct}.
		$$
	\end{itemize}
}

\smallskip 

The Section~\ref{sec3} deals with the case when there is an additional information about the combinatorics of a generalized polyhedron. Namely, in the Theorem~\ref{theoremTrianglesBounds} the upper bounds for volumes are obtained by taking into account the number of triangular faces and trivalent vertices of the polyhedron. 

\smallskip
\noindent
\textbf{Theorem~3.4.} \emph{Let $\Gamma$ be a 3-connected planar graph with $E$ edges, and $P$ be a generalized hyperbolic polyhedron for which $\Gamma$ is the 1-skeleton.
	\begin{itemize}
		\item[(a)] If $P$ has $V_3$ trivalent vertices and $p_3$ triangular faces, then
		$$
		\volume(P) \leqslant 2v_{tet} \cdot \left(E - \frac{p_3+V_3+8}{4} \right).
		$$
		\item[(b)] If all vertices of $P$ are trivalent and there are $p_3$ triangular faces, then
		$$
		\volume(P) \leqslant \frac{5v_{tet}}{3} \left( E - \frac{3 p_3 + 24}{10} \right).
		$$
	\end{itemize}
}
\smallskip 

In Section~\ref{sec4} we provide  examples of applying bounds from  Theorems~\ref{theoremPolyhedraBounds} and~\ref{theoremTrianglesBounds} to three infinite fa\-mi\-li\-es of generalized hyperbolic polyhedra: pyramids, prisms and pyramids with two apexes. In Section~\ref{sec5} we present the relationship between the volumes of hyperbolic polyhedra and bounds for the volumes of hyperbolic knots and links via the number of twists in their diagrams. Relations of such type were previously discussed in~\cite{Ad17-1, DT15, La04, Pu20}. In the Theorem~\ref{theoremKnotsBounds} we obtain an upper bound for the volumes of hyperbolic knots and links with the number of twists in the diagram greater than eight. 

\smallskip
\noindent
\textbf{Theorem~5.1.} \emph{Let $D$ be a hyperbolic diagram of a link $K$ with $t(D)$ twists. If $t(D)>$8, then}
	$$
	\volume{(S^3 \setminus K)} \leq 10 v_{tet} \cdot (t(D)-1.4). 
	$$

Finally we demonstrate that the bound from Theorem~5.1 improves the previously known bounds. 
	
\section{Volume of a generalized hyperbolic polyhedron} \label{sec2} 

To define a generalized hyperbolic polyhedron we will use a projective model of a hyperbolic space and follow~\cite{BB02, Be21, Th80, Us06}.   Consider the symmetric bilinear form defined on $\mathbb R^4$ as 
$$
\langle {\bf x}, {\bf y} \rangle = -x_0 y_0 + x_1 y_1 + x_2 y_2 + x_3 y_3.
$$ 
With the standard embedding of $\mathbb R^3$ in $\mathbb {RP}^3$, which maps the point $(x_1, x_2, x_3) \in \mathbb R^3$ to the point in $\mathbb {RP}^3$ with homogeneous coordinates $(1, x_1, x_2, x_3)$, a subset $\mathbb H^3$ corresponds to an open unit ball in $\mathbb R^3$. At the same time, geodesics in $\mathbb H^3$ are intersections of $\mathbb H^3$ with projective lines from $\mathbb{RP}^3$ or, equivalently, with lines from $\mathbb R^3\subset \mathbb{RP}^3$. Similarly, the (totally geodesic) hyperbolic planes in $\mathbb H^3$ correspond to nonempty intersections of $\mathbb H^3$ and projective planes from $\mathbb{RP}^3$, or equivalently,  with affine planes from $\mathbb R^3$.

In the projective model of the hyperbolic space $\mathbb H^3$, the following duality holds. For a $k$-dimensional, $0\leq k\leq2$, projective subspace $\ell\subset\mathbb{RP}^3$, consider the corresponding $(k+1)$-dimensional linear subspace $L\subset\mathbb R^4$. Then the subspace $L^{\perp}$, orthogonal to $L$ with respect to the form $\langle\,  \, ,  \, \,\rangle$ introduced above, is a $(3-k)$-dimensional linear subspace in $\mathbb R^4$ and defines $(2-k)$-dimensional projective subspace $\ell^{\perp} \subset \mathbb {RP}^3$. In particular, if $x\in\mathbb{RP}^3\setminus\overline{\mathbb H}^3$, then $x^{\perp}$ is a plane that intersects $\mathbb H^3$, and the point $x$ is called \emph{hyperideal}.

The realization of a convex Euclidean polyhedron in the projective model of the space $\mathbb H^3$ will be called a \emph{generalized hyperbolic polyhedron} if each of its vertices is finite, ideal or hyperideal. In this case, each edge of the polyhedron must contain internal points of the hyperbolic space. To each hyperideal point $p$ we assign a \emph{polar plane} $\Pi_p\subset \mathbb H^3$, which is a plane orthogonal to all lines passing through $\mathbb H^3$ and $p$. The plane $\Pi_p$ divides $\mathbb H^3$ into two half-spaces, denote by $H_p \subset \mathbb H^3$ the one that contains $0\in \mathbb R^3$. A generalized hyperbolic polyhedron $P$ will be called \emph{proper} if for each hyperideal vertex $v$ of the polyhedron $P$ the interior of the half-space $H_v$ contains all the finite vertices of the polyhedron $P$.
	
Let $P$ be a generalized hyperbolic polyhedron and $U(P)$ be the set of all its hyperideal vertices. We define \emph{truncation} $P_{tr}$ of a generalized hyperbolic polytope $P$ as the following set:
$$
P_{tr}=P \bigcap_{v \in U(P)} H_v.
$$
Then \emph{volume of the generalized polyhedron} $P$ is defined as the volume of its truncation $P_{tr}$. Note that if the polyhedron $P$ is proper, then the dihedral angles at the new edges arising after truncation are equal to $\pi/2$.

Following~\cite{Be21} we will say a polyhedron $\overline{\Gamma} \subset\mathbb R^3\subset\mathbb{RP}^3$ is a \emph{rectification} of a 3-connected planar graph $\Gamma$ if the 1-skeleton  $\overline{\Gamma}$ coincides with  $\Gamma$ and all edges of $\overline{\Gamma}$ are tangent to $\partial H^3$. Notice that $\overline{\Gamma}$ is not a generalized hyperbolic polyhedron since none of its edges intersect $\mathbb H^3$. Nevertheless, for $\overline{\Gamma}$, it is possible, as above, to define a truncation $\overline{\Gamma}_{tr}$, which will be an ideal right-angled polyhedron whose 1-skeleton is the medial graph for $\Gamma$. By the volume $\operatorname{vol} (\overline{\Gamma})$ of the rectification $\overline{\Gamma}$ we will understand the volume $\operatorname{vol} (\overline{\Gamma}_{tr})$ of its truncation $\overline{\Gamma}_{tr}$.

\smallskip 

In~\cite[Corollary~10]{At11} Atkinson obtained the following upper bound. Let $P$ be a non-obtuse  hyperbolic polyhedron cantaining $V_3$ trivalent vertices and $V_4$ quadrivalent vertices. Then
\begin{equation}
\volume(P) < \frac{2 V_4 + 3 V_3 - 2}{4} \cdot v_{oct} + \frac{15 V_3 + 20 V_4}{16} \cdot v_{tet}. \label{eqn1}
\end{equation}

In~\cite{Be21} Beletti established that the volume of an arbitrary generalized hyperbolic polyhedron can be estimated from above by the volume of an ideal right-angled hyperbolic polyhedron constructed from its 1-skeleton.

\smallskip
	
\begin{theorem} {\rm \cite[Theorem~4.2]{Be21}} \label{theoremBeletti}
	For any 3-connected planar graph $\Gamma$, 
	$$
	\sup_{P} \volume (P) = \volume (\overline{\Gamma}),
	$$
	where $P$ varies among all proper generalized hyperbolic polyhedra with 1-sceleton $\Gamma$ and 
	$\overline{\Gamma}$ is the rectification of  $\Gamma$.
\end{theorem}

By definition, the volume of the rectification $\overline{\Gamma}$ is equal to the volume of the polyhedron $\overline{\Gamma}_{tr}$ that is  an ideal right-angled hyperbolic polyhedron such that its 1-skeleton is the medial graph of the graph $\Gamma$. By construction, all vertices of $\overline{\Gamma}_{tr}$ are quadrivalent. 
Recall that if $G$ is a plane embedding of a graph then \emph{medial graph} for it is a graph $M(G)$ such that the vertices of $M(G)$ correspond one-to-one to the edges of $G$ and for each face $G$ if two edges in it go sequentially then the corresponding vertices from $M(G)$ are connected by an edge.

The initial list of ideal right-angled polyhedra is presented in~\cite{EV20-1}, where the first $248$ values of the volumes of such polyhedra are also computed. A well-known infinite family of ideal right-angled polyhedra is the family of $n$-antiprisms for integers $n\geq 3$. In particular, the $3$-antiprism is an octahedron. The formula for the volumes of ideal $n$-antiprisms with cyclic symmetry was obtained by Thurston~\cite{Th80} in connection with the calculation of the volumes of the family of chain links. The arithmeticity of the groups generated by reflections in the faces of ideal right-angled antiprisms (and, consequently, the arithmeticity of the groups of the corresponding chain links) was investigated in papers~\cite{Ke22} and~\cite{MMT20}.

Two-sided bounds for the volumes of ideal right-angled hyperbolic polyhedra in terms of the number of their vertices were obtained by Atkinson~\cite[Theorem~2.2]{At09}. Namely, if $P$ is an ideal right-angled hyperbolic polyhedron with $V$ vertices, then
\begin{equation}
\frac{v_{oct}}{4} \cdot V - \frac{v_{oct}}{2} \leqslant \volume (P) \leqslant \frac{v_{oct}}{2} \cdot V - 2 v_{oct}. \label{eqn2}
\end{equation}
At the same time, both inequalities turn into equalities when $P$ is an ideal right-angled octahedron, that is, when $V=6$.

An ideal right-angled octahedron is the unique ideal right-angled polyhedron with $V=6$, and its volume is $v_{oct}$.  The next ideal right-angled polyhedra have $V\geq 8$ vertices, and the upper bound can be improved. Namely, it is shown in~\cite[Theorem~2.3]{EV20-2} that if $P$ is an ideal right-angled hyperbolic polyhedron with $V\geq 8$ vertices, then
\begin{equation}
\volume (P) \leqslant \frac{v_{oct}}{2} \cdot V - \frac{5v_{oct}}{2}. \label{eqn3}
\end{equation}
	
The volumes of polyhedra with the number of vertices $V\leq 21$ were tabulated in~\cite{EV20-1}. Then it was shown in \cite[Theorem~1.3]{ABEV}  that the upper bound (\ref{eqn2}) can be improved if we don't consider polyhedra with $V\leq 24$ vertices. Namely, by virtue of ~\cite[Theorem~2.3]{ABEV}, if $P$ is an ideal right-angled hyperbolic polyhedron with $V > 24$ vertices, then
	\begin{equation}
	\volume (P) \leqslant \frac{v_{oct}}{2} \cdot V - 3v_{oct}. \label{eqn4}
	\end{equation}

\smallskip 
		
\begin{theorem} \label{theoremPolyhedraBounds}
	Let $\Gamma$ be a 3-connected planar graph with $E$ edges, and $P$ be a generalized hyperbolic polyhedron for which $\Gamma$ is a 1-skeleton. Then the following inequalities hold.
	\begin{itemize}
		\item[(a)] If $P$ is a tetrahedron, then $\volume(P) \leq v_{oct}$.
		\item[(b)] If $P$ is not a tetrahedron, then
		$$
		\volume (P) \leq \frac{v_{oct}}{2} \cdot E - \frac{5 v_{oct}}{2}.
		$$
		\item[(c)] If the number of edges $E > 24$, then
		$$
		\volume (P) \leq \frac{v_{oct}}{2} \cdot E - 3 v_{oct}.
		$$
	\end{itemize}
\end{theorem}

\begin{proof}
	It follows from the Theorem~\ref{theoremBeletti} and the formulae (\ref{eqn2}), (\ref{eqn3}), (\ref{eqn4}).
\end{proof}

It is well known, see, for example,~\cite{Br05}, that for every ideal right-angled polyhedron its skeleton is the medial graph for two polyhedra combinatorially dual to each other.
		
\section{Polyhedra with trivalent vertices and triangular faces} \label{sec3}

Note that if the polyhedron $P$ has some special combinatorial properties, then the upper bound for its volume can be improved. In this section, we will present improvements in the case when the information about the numbers of trivalent vertices and triangular faces is used.

First of all, we consider the \emph{regular} ideal $n$-gonal bipyramid $B_n^r$, $n\geq 3$, see~\cite{Ad17-1}. Regular means that $B_n^r$ is obtained by gluing together $n$ copies of an ideal tetrahedron $T_n$ around a common edge, where $T_n$ is given by the dihedral angles $\frac{2\pi}{n}$, $\frac{(n-2)\pi}{2n}$ and $\frac{(n-2)\pi}{2n}$ for edges incident to one of the vertices and the requirement that the dihedral angles for opposite edges of the tetrahedron are equal. That is, following the notation for ideal hyperbolic tetrahedra from~\cite{Mi82}, we can write that $T_n = T(\frac{2\pi}{n}, \frac{\pi}{2} - \frac{\pi}{n}, \frac{\pi}{2} - \frac{\pi}{n})$. As shown in ~\cite[Theorem~2.1]{Ad17-1}, the maximum volume of an ideal $n$-bipyramide is reached when it is regular. The formula for the volume of the tetrahedron $T_n$ is given in ~\cite{Ad17-1} in the following form:
$$
\volume(T_n) = \int_{0}^{2\pi/n} -2\ln(\sin\theta) d\theta + 2\int_{0}^{\pi (n-1)/2n} - 2\ln(\sin\theta) d\theta.
$$
By~\cite{Mi82}, this volume can also be written in terms of the Lobachevsky function as follows:
$$
\volume(T_n) = \Lambda \left( \frac{2\pi}{n} \right) + 2 \Lambda \left( \frac{\pi}{2} - \frac{\pi}{n} \right) = 2 \Lambda \left( \frac{\pi}{n} \right),
$$
where we used the identities $\Lambda(2x) = 2\Lambda(x) + 2\Lambda (x +\frac{\pi}{2})$ and $\Lambda(-\theta) = - \Lambda (\theta)$.
Thus,
$$
\volume (B_n^r) = 2n \Lambda \left( \frac{\pi}{n} \right).
$$
Below we will use this equality to estimate the volume of an ideal right-angled polyhedron.

\smallskip

\begin{lemma} \label{lemma4.1}
	Let $P$ be an ideal right-angled hyperbolic polyhedron. Denote by $p_n$, $n\geq 3$, the number of its $n$-gonal faces. Then
	\begin{equation}
	\volume(P) \leqslant \sum_{n \geq 3} \Lambda \left( \frac{\pi}{n} \right) p_n n - 4v_{tet}. \label{eqn5}
	\end{equation}
\end{lemma}

\begin{proof}
Denote by $\partial P$ the surface of the polyhedron $P$, which naturally splits into polygons corresponding to the faces of $P$. Let us choose  a vertex $v$ of $P$ and connect $v$ with other vertices of $P$ by geodesic lines. Thus, we obtain a subdivision of $P$ into pyramids with apex $v$ over polygons splitting of $\partial P$. For each resulting $n$-gonal pyramid, consider its double, which is an ideal $n$-gonal bipyramid. Since the maximum volume of an ideal $n$-gonal bipyramid is reached when it is regular~\cite[Theorem~2.1]{Ad17-1}, the volume of each of the $n$-gonal pyramids under consideration is bounded by $\frac{1}{2}\volume(B_n^r)$, where, as well as above, the regular $n$-gonal bipyramid is denoted by $B_n^r$. Since $\volume (B_n^r) = 2n\Lambda\left(\frac{\pi}{n}\right)$, we get
$$
\volume(P) \leqslant \sum_{n \geq 3} \Lambda \left( \frac{\pi}{n} \right) p_n n.
$$
Under the construction, four pyramids, based on the faces incident to $v$, degenerate. Their contribution to the volume bound was no less than the sum of the volumes of four regular ideal tetrahedra, since  
$$
4\cdot\frac{1}{2}\volume(B_3^r) = 4\cdot 3\Lambda\left( \frac{\pi}{3}\right) =4 \cdot v_{tet}.
$$
Thus, the inequality (\ref{eqn5}) is obtained.
\end{proof}

By~\cite[Theorem~2.2]{Ad17-1}, there is a bound $\operatorname{vol}(B_n^r) \leq 2 \pi\ln(n/2)$ for $n\geq 3$, with $\operatorname{vol} (B_n^r)$ growing asymptotically as $2\pi\ln(n/2)$ for $n\to\infty$. Using this bound for the volume of a regular bipyramid along with the inequality (\ref{eqn5}), we obtain the following result. 

\smallskip

\begin{corollary} \label{cor4.2}
Let $P$ be an ideal right-angled hyperbolic polyhedron. Denote by $p_n$, $n\geq 3$, the number of its $n$-gonal faces. Then
\begin{equation}
\volume(P) \leqslant \pi \sum_{n \geq 3} \ln \left( \frac{n}{2} \right) p_n - 4v_{tet}. \label{eqn6}
\end{equation}
\end{corollary}

\smallskip 

Let $P$ be an ideal right-angled hyperbolic polyhedron, and $p_n$, $n\geq 3$, denote the number of its $n$-gonal faces. From Euler's formula for polyhedra and from the quadrivalence of the vertices of the polyhedron $P$ follows, see for example~\cite{EV20-2}, that
$$
p_3 = 8 +\sum_{k\geq 5} (k-4) p_k.
$$
Hence $P$ has at least eight triangular faces. The following lemma gives the bounds for the volume of an ideal right-angled polyhedron when the information about the number of triangular faces is used.

\smallskip

\begin{lemma} \label{lemma4.2}
	Let $P$ be an ideal right-angled hyperbolic polyhedron with $V$ vertices and $p_3$ triangular faces. Then
	\begin{itemize}
		\item[(a)] The following inequality holds: 
		$$
		\volume(P) \leqslant 2 v_{tet} \left(V - \frac{p_3 +8}{4}\right).
		$$
		\item[(b)] If $V>24$, then
		$$
		\volume(P) \leqslant 2 v_{tet} \left( V - \frac{p_3+13}{4} \right).
		$$
	\end{itemize}
\end{lemma}

\begin{proof}
	(a) Let $F$ be the number of faces of the polyhedron $P$ and denote faces by $f_1, \ldots, f_F$.
	Similar to the proof of Lemma~\ref{lemma4.1}, we consider the decomposition of the polyhedron $P$ into ideal pyramids $\tau_i$, $i=1, \ldots, F$,  such that the face $f_i$ is the base of $\tau_i$ and all pyramids have a common apex $v$. For each ideal pyramid $\tau_i$, consider its doubling,  the ideal bipyramid $\beta_i$. Hence,
	$$
	\operatorname{vol} (P) = \frac{1}{2} \sum_{i=1}^{F} \operatorname{vol} (\beta_i).
	$$
	
	Let for certainty $\tau_i$ be pyramid for some $n\geq 3$. Then $\beta_i$ is an $n$-gonal bipyramid. Let us split $\beta_i$ into ideal tetrahedra. If $n=3$, then the pyramid $\tau_i$ is a tetrahedron and $\beta_i$ is the union of two ideal tetrahedra along a common face, whence $\operatorname{vol} (\beta_i) \leq 2 v_{tet}$. If $n\geq 4$, then $\beta_i$ is splittable into $n$ ideal tetrahedra having a common edge that contains the apex $v$ and its double $v'$, see the example for $n=4$ shown in Figure ~\ref{fig1}. Thus, the volume of the $n$-gonal bipyramide is bounded by $2v_{tet}$ if $n=3$, and by $n v_{tet}$ if $n\geq 4$.
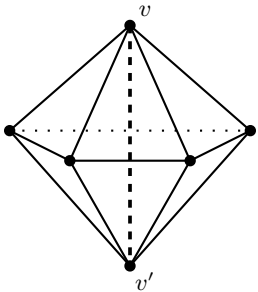
\begin{figure}[ht]
\begin{center}
\unitlength=.1mm
\begin{tikzpicture}[scale=0.4] 
%
\draw [line width=2pt, black] (0,9) circle[radius=0.1cm];
\draw [line width=2pt, black] (0,1) circle[radius=0.1cm];
\draw [line width=2pt, black] (-2,4.5) circle[radius=0.1cm];
\draw [line width=2pt, black] (2,4.5) circle[radius=0.1cm];
\draw [line width=2pt, black] (4,5.5) circle[radius=0.1cm];
\draw [line width=2pt, black] (-4,5.5) circle[radius=0.1cm];
\draw [line width=1.4pt, dashed, black] (0,9) -- (0,1);
\draw [thick, black] (0,9)-- (-4,5.5) -- (0,1);
\draw [thick, black] (0,9)-- (-2,4.5) -- (0,1);
\draw [thick, black] (0,9)-- (4,5.5) -- (0,1);
\draw [thick, black] (0,9)-- (2,4.5) -- (0,1);
\draw [thick, black] (-4,5.5)-- (-2,4.5) -- (2,4.5) -- (4,5.5);
\draw [thick, loosely dotted, black] (-4,5.5) -- (4,5.5);
\node (K) at (0.5,9.5) {\small $v$};
\node (K) at (0.5,0.5) {\small $v'$};
\end{tikzpicture}
\end{center}
\caption{Splitting an ideal 4-bipyramid into 4 ideal tetrahedra.} \label{fig1}
\end{figure}

Denote by $E$ the number of edges of the polyhedron $P$. Since all dihedral angles of the polyhedron $P$ are equal to $\pi/2$, each edge $e\in E$ is incident to two bipyramids. Thus, the sum of incident edges over all bipyramids will be equal to twice the number of edges $2E$ (each edge incident to the base of the pyramid in the splitting of $P$ is counted twice). Since the sum of incident edges on triangular bipyramids is equal to $3 p_3$, the sum of incident edges corresponding to $n$-gonal bipyramids, $n\geq 4$, is equal to $2E - 3 p_3$. Thus,
\begin{equation}
2 \operatorname{vol} (P) \leq 2 v_{tet} \cdot p_3 + v_{tet} \cdot (2E - 3p_3) = v_{tet} \cdot (4V - p_3), \label{eqn7}
\end{equation}
where we used $E = 2 V$ since each vertex of $P$ is quadrivalent.

In the bound (\ref{eqn7}), we did not take into account that the apex $v$ is incident to four faces and so, four pyramids, and therefore four bipyramids degenerate into flat ones. The contribution of four bipyramids in (\ref{eqn7}) is at least $4\cdot (2 v_{tet})$, that  is the case when all four bipyramids are $3$-bipyramids. Therefore, 
$$
2 \operatorname{vol} (P) \leq v_{tet} \cdot (4V - p_3 - 8),
$$
and so  
$$
\operatorname{vol}(P) \leq 2 v_{tet} \cdot \left(V - \frac{p_3 + 8}{4}\right).
$$

(b) Let us choose the common apex of the pyramids in a special way. Obviously, for each vertex of $P$ there are four adjacent vertices. Following~\cite{ABEV}, we will say that two vertices are \emph{quasi-adjacent} if they not adjacent, but  belong to the same face.  According to~\cite[Lemma~2.1]{ABEV}, if ideal right-angled polyhedron has $V>24$ vertices, there is a vertex $v_0$ that is quasi-adjacent to at least four vertices. Since $v_0$ has 4 adjacent vertices, we get that $v_0$ is adjacent to four faces such that the sum of their sides is at least $16$. Taking a splitting of $P$  into pyramids with a common apex $v_0$ we will get that at least $13$ tetrahedra will degenerate. The bound holds by the same arguments as in the item (a).
\end{proof}

\smallskip 

Now we are ready to present the volume bounds which improve Theorem~\ref{theoremPolyhedraBounds} in the case when there is the additional information about numbers of trivalent vertices and triangular faces. 

\smallskip

\begin{theorem} \label{theoremTrianglesBounds}
	Let $\Gamma$ be a 3-connected planar graph with $E$ edges, and $P$ be a generalized hyperbolic polyhedron for which $\Gamma$ is the 1-skeleton.
	\begin{itemize}
		\item[(a)] If $P$ has $V_3$ trivalent vertices and $p_3$ triangular faces, then
		$$
		\volume(P) \leqslant 2v_{tet} \cdot \left(E - \frac{p_3+V_3+8}{4} \right).
		$$
		\item[(b)] If all vertices of $P$ are trivalent and there are $p_3$ triangular faces, then
		$$
		\volume(P) \leqslant \frac{5v_{tet}}{3} \left( E - \frac{3 p_3 + 24}{10} \right).
		$$
	\end{itemize}
\end{theorem}

\begin{proof}
	(a) We use notations $V$, $E$ and $F$ for the number of vertices, edges and faces of the graph $\Gamma$, and similarly, $\overline{V}$, $\overline{E}$ and $\overline{F}$ for the number of vertices, edges and the faces of the 1-skeleton of the polyhedron $\overline{\Gamma}_{tr}$. Since the 1-skeleton of  $\overline{\Gamma}_{tr}$ is the medial graph for $\Gamma$, we have $\overline{V}= E$, $\overline{E}= 2\overline{V}= 2 E$ and $\overline{F} = V + F$. If a face of $\overline{\Gamma}_{tr}$ corresponds to a vertex of $\Gamma$, then the number of its sides is equal to the valence of the vertex. If a face of $\overline{\Gamma}_{tr}$ corresponds to a face of $\Gamma$, then the number of its sides is equal to the number of sides in the original face. Therefore, $\overline{\Gamma}_{tr}$ has $V_3+p_3$ triangular faces. Applying Lemma~\ref{lemma4.1} to the polyhedron $\overline{\Gamma}_{tr}$, we obtain the required inequality.
	
	(b) If all vertices of $\Gamma$ are trivalent, then $2 E = 3 V = 3 V_3$. By substituting $V_3 = \frac{2}{3} E$ into the estimate from the item (a), we get:
	$$
	\volume{P} \leq 2 v_{tet} \cdot \left(E - \frac{p_3+V + 8}{4} \right) = \frac{5 v_{tet}}{3} \left(E - \frac{3p_3 + 24}{10} \right),
	$$
	that completes the proof.
\end{proof}

\smallskip

\begin{remark}
	{\rm
		To compare the bounds obtained in Theorems~\ref{theoremPolyhedraBounds} and~\ref{theoremTrianglesBounds}, we note that the inequality $\frac{5}{3} v_{tet} < \frac{1}{2} v_{oct}$ holds. Thus, if all the vertices of the polyhedron are trivalent, then the formulae from the Theorem~\ref{theoremTrianglesBounds} give the better asymptotics. Moreover, the inequality
		$$
		\frac{5 v_{tet}}{3} \left(E - \frac{3p_3 + 24}{10} \right) < \frac{v_{oct}}{2} \cdot E - 3 v_{oct}
		$$
		is equivalent to the inequality
		$$
		E + \frac{3 v_{tet}}{3v_{oct} - 10v_{tet} } p_3 > \frac{6 (3 v_{oct} - 4 v_{tet})}{3v_{oct} - 10 v_{tet}}.
		$$
		Using the approximate values $v_{tet} = 1.014941$ and $v_{oct} = 3.663863$, we obtain that for
		$$
		E + 3.615410 \cdot p_3 > 49.385163
		$$
		the bound from the Theorem~\ref{theoremTrianglesBounds} is stronger.
	}
\end{remark}
\section{Pyramids, prisms and two-apex pyramids} \label{sec4}

In this section we give some examples of calculating the upper bounds based on the above obtained formulas and compare them with the known volume values.

\smallskip

\textbf{5.1. Pyramids.}
Note that the medial graph for the $1$-skeleton of the $n$-gonal pyramid $P_n$ is the $1$-skeleton of the $n$-antiprism $A(n)$, see Figure~\ref{fig2} for $n=4$. Recall that the $n$-antiprism $A(n)$ is an ideal right-angled polyhedron with $2n$ $4$-valent vertices, $(2n+2)$ faces, upper and lower $n$-angular bases and a side surface of two layers of $n$ triangles.
\begin{figure}[ht]
\begin{center}
\unitlength=.1mm
\begin{tikzpicture}[scale=0.4] 
\unitlength=10.mm
\draw [line width=2pt, black] (-5,2) circle[radius=0.1cm];
\draw [line width=2pt, black] (-11,2) circle[radius=0.1cm];
\draw [line width=2pt, black] (-5,8) circle[radius=0.1cm];
\draw [line width=2pt, black] (-11,8) circle[radius=0.1cm]; 
\draw [line width=2pt, black] (-8,5) circle[radius=0.1cm]; 
\draw [thick, black] (-11,2)-- (-11,8) -- (-5,8) -- (-5,2) -- (-11,2);
\draw [thick, black,->] (-11,2)-- (-5,8);
\draw [thick, black,->] (-11,8)-- (-5,2);
\draw [line width=2pt, black] (4,5) circle[radius=0.1cm];
\draw [line width=2pt, black] (12,5) circle[radius=0.1cm];
\draw [line width=2pt, black] (8,1) circle[radius=0.1cm];
\draw [line width=2pt, black] (8,9) circle[radius=0.1cm];
\draw [thick, black] (4,5)-- (8,9) -- (12,5) -- (8,1) -- (4,5);
\draw [line width=2pt, black] (7,4) circle[radius=0.1cm];
\draw [line width=2pt, black] (9,4) circle[radius=0.1cm];
\draw [line width=2pt, black] (7,6) circle[radius=0.1cm];
\draw [line width=2pt, black] (9,6) circle[radius=0.1cm];
\draw [thick, black] (7,4)-- (7,6) -- (9,6) -- (9,4) -- (7,4);
\draw [thick, black] (7,4)-- (4,5) -- (7,6) -- (8,9) --(9,6) -- (12,5) -- (9,4) -- (8,1) -- (7,4);
\end{tikzpicture}
\end{center}
\caption{Pyramid $P_4$ and antiprism $A_4$.} \label{fig2}
\end{figure}
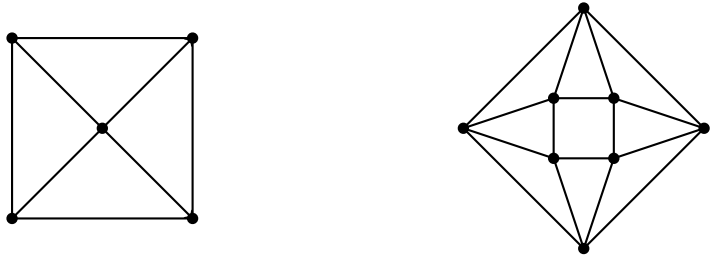

Let $n=3$. The triangular pyramid $P_3$ is a tetrahedron with $E=6$ edges, $p_3= 4$ triangular faces, and $V_3= 4$ trivalent vertices.
From the case (a) of the Theorem~\ref{theoremTrianglesBounds} we get
$$
\operatorname{vol} (P_3) < 2 \, v_{tet} \, \left( 6 - \frac{4+4+8}{4}\right) = 4\, v_{tet}.
$$
Recall that the formula for the volume of a generalized hyperbolic tetrahedron was given in~\cite{Us06}. It is well known that the medial graph for the 1-skeleton of a tetrahedron is the 1-skeleton of an octahedron. The volume of an ideal right-angled octahedron is $v_{oct} = 3.663863$. Thus, the volume of any generalized tetrahedron does not exceed $v_{oct}$. This fact was noted in~\cite{Us06}. Since $v_{oct} <4\, v_{tet}$, the estimate obtained from the Theorem~\ref{theoremTrianglesBounds} is correct, although it is not accurate.

For $n\geq 4$, the pyramid $P_n$ has $E = 2n$ edges, $p_3= n$ triangular faces, and $V_3= n$ trivalent vertices.
From the case (a) of the Theorem~\ref{theoremTrianglesBounds} we get
$$
\operatorname{vol} (P_n) \leq 2 \, v_{tet} \left( 2n - \frac{2n+8}{4} \right) = 3 v_{tet} \cdot n - 4 v_{tet}.
$$
Note that approximately $3 v_{tet} = 3.044823$.

Recall that the formula for the volume of an ideal right-angled hyperbolic antiprism $A(n)$, $n\geq 3$, is known. It was obtained by Thurston in~\cite{Th80}:
$$
\volume(A(n)) = 2 n \, \left[ \Lambda \left( \frac{\pi}{4} + \frac{\pi}{2n} \right) + \Lambda \left( \frac{\pi}{4} - \frac{\pi}{2n} \right) \right].
$$

Thus, for $n\geq 4$, the volume of the generalized hyperbolic $n$-pyramid $P_n$ satisfies the following inequality:
$$
\operatorname{vol} (P_n) \leq 2 n \, \left[ \Lambda \left( \frac{\pi}{4} + \frac{\pi}{2n} \right) + \Lambda \left( \frac{\pi}{4} - \frac{\pi}{2n} \right) \right] .
$$
The right side of the inequality is asymptotically equivalent to $\frac{1}{2} v_{oct}\cdot n$ as $n\to\infty$.

\smallskip

\textbf{5.2. Prisms.} Denote by $\Pi_n$ a generalized hyperbolic $n$-gonal prism for $n\geq 3$, that is, a polyhedron having upper and lower $n$-gonal bases and $n$ quadrangular faces on the lateral surface. The vertices of a polyhedron can be finite, ideal, or hyperideal, and the dihedral angles are such that the polyhedron can be realized in $\mathbb H^3$.

Let $n=3$. The prism $\Pi_3$ is a triangular prism that has $E=9$ edges, $p_3= 2$ triangular faces, and $V_3=6$ trivalent vertices. From the case (a) of the Theorem~\ref{theoremTrianglesBounds} we get that
$$
\operatorname{vol} (\Pi_3) < 2 \, v_{tet} \cdot \left(9 - \frac{2 + 6 - 8}{4} \right) = 18 \, v_{tet}.
$$
At the same time, the rectification $\Pi_3$ is a polyhedron composed of two octahedra, whence $\operatorname{vol}(\Pi_3)\leq 2v_{oct}$. Since $2 v_{oct} <18\, v_{tet}$, the estimate obtained from the Theorem~\ref{theoremTrianglesBounds} is correct, although it is not accurate.

Let $n=4$. The prism $\Pi_4$ is a cube. It is easy to see that the medial graph for the 1-skeleton of a cube is the 1-skeleton of an ideal right-angled polyhedron $Q_{14}$ with $8$ triangular and $6$ quadrangular faces, shown in Figure~\ref{fig3}.
\begin{figure}[ht]
\begin{center}
\unitlength=.1mm
\begin{tikzpicture}[scale=0.4] 
\unitlength=10.mm
\draw [line width=2pt, black] (-5,2) circle[radius=0.1cm];
\draw [line width=2pt, black] (-11,2) circle[radius=0.1cm];
\draw [line width=2pt, black] (-5,8) circle[radius=0.1cm];
\draw [line width=2pt, black] (-11,8) circle[radius=0.1cm]; 
\draw [thick, black] (-11,2)-- (-11,8) -- (-5,8) -- (-5,2) -- (-11,2);
\draw [line width=2pt, black] (-9.5,3.5) circle[radius=0.1cm];
\draw [line width=2pt, black] (-9.5,6.5) circle[radius=0.1cm];
\draw [line width=2pt, black] (-6.5,6.5) circle[radius=0.1cm];
\draw [line width=2pt, black] (-6.5,3.5) circle[radius=0.1cm]; 
\draw [thick, black] (-9.5,3.5)-- (-9.5,6.5) -- (-6.5,6.5) -- (-6.5,3.5) -- (-9.5,3.5);
\draw [thick, black] (-11,2)-- (-9.5,3.5);
\draw [thick, black] (-11,8)-- (-9.5,6.5);
\draw [thick, black] (-5,8)-- (-6.5,6.5);
\draw [thick, black] (-5,2)-- (-6.5,3.5);
\draw [line width=2pt, black] (4,5) circle[radius=0.1cm];
\draw [line width=2pt, black] (12,5) circle[radius=0.1cm];
\draw [line width=2pt, black] (8,1) circle[radius=0.1cm];
\draw [line width=2pt, black] (8,9) circle[radius=0.1cm];
\draw [thick, black] (4,5)-- (8,9) -- (12,5) -- (8,1) -- (4,5);
\draw [line width=2pt, black] (6.5,3.5) circle[radius=0.1cm];
\draw [line width=2pt, black] (9.5,3.5) circle[radius=0.1cm];
\draw [line width=2pt, black] (6.5,6.5) circle[radius=0.1cm];
\draw [line width=2pt, black] (9.5,6.5) circle[radius=0.1cm];
\draw [thick, black] (6.5,3.5)-- (4,5) -- (6.5,6.5) -- (8,9) --(9.5,6.5) -- (12,5) -- (9.5,3.5) -- (8,1) -- (6.5,3.5);
\draw [line width=2pt, black] (7,5) circle[radius=0.1cm];
\draw [line width=2pt, black] (9,5) circle[radius=0.1cm];
\draw [line width=2pt, black] (8,4) circle[radius=0.1cm];
\draw [line width=2pt, black] (8,6) circle[radius=0.1cm];
\draw [thick, black] (6.5,3.5)-- (7,5) -- (6.5,6.5) -- (8,6) -- (9.5,6.5) -- (9,5) -- (9.5,3.5) -- (8,4) -- (6.5,3.5);
\draw [thick, black] (7,5) -- (8,6) -- (9,5) -- (8,4) -- (7,5);
\end{tikzpicture}
\end{center}
\caption{Prism $\Pi_4$ and polyhedron $Q_{14}$.}\label{fig3}
\end{figure}
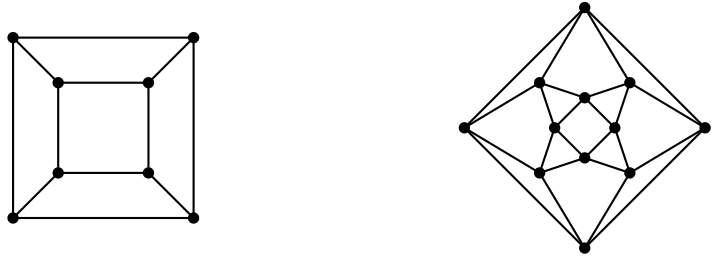
Its volume is calculated in~\cite{EV20-1} and is approximately equal to $12.046092$. This polyhedron is the union of two copies of the ideal antiprism $A(4)$ along a quadrangular face. Note also that this polyhedron has the maximum volume among all nine ideal right-angled hyperbolic polyhedra with $14$ faces.

Let us have a unified discussion of the case $n\geq 4$.
In~\cite[Corollary~11]{At11} the following inequality for the volume of the prism $\Pi_n$ was obtained:
\begin{equation}
\volume(\Pi_n) < \frac{3}{2} \, v_{oct} \cdot n - 2 \, v_{oct}. \label{eqn8}
\end{equation}
Note that $\frac{3}{2} v_{oct} = 5.495794$.

Since all the vertices of the prism $\Pi_n$ are trivalent and its 1-skeleton has $E = 3n$ edges, from the case (b) of the Theorem~\ref{theoremTrianglesBounds} we obtain
\begin{equation}
\volume(\Pi_n) < 5 \, v_{tet} \cdot n - 4 \, v_{tet}. \label{eqn9}
\end{equation}
Note that $5 v_{tet} = 5.074705$.

\smallskip
\begin{remark} {\rm
		Inequality
		$$
		5 \, v_{tet} \cdot n - 4 \, v_{tet} < \frac{3}{2} \, v_{oct} \cdot n - 2 \, v_{oct}
		$$
		occurs when
		$$
		n > \frac{2 v_{oct} - 4 v_{tet}}{ \frac{3}{2} v_{oct} - 5 v_{tet}}.
		$$
		Substituting the specified constants by their approximate values, we get that the bound (\ref{eqn9}) improves the bound (\ref{eqn8}) for $n > 7.760616$, that is, for $n\geq 8$.
	}
\end{remark}

The volume formula of an ideal right-angled hyperbolic $n$-antiprism $A(n)$ was obtained in~\cite{Th80}. Using this formula, we get
$$
\volume(\Pi_n) < 2 \volume (A(n)) = 4n \left[ \Lambda \left( \frac{\pi}{4} + \frac{\pi}{2n} \right) + \Lambda \left( \frac{\pi}{4} - \frac{\pi}{2n} \right) \right].
$$
At the same time, $\volume(A(n))$ is asymptotically equivalent to $\frac{1}{2} v_{oct}\cdot n$ for $n\to\infty$.

\smallskip

\textbf{5.3. Two-apex pyramids.}
Consider the polyhedron $W_n$, $n\geq 4$, which is obtained from the pyramid $P_n$ as follows. Split the apex of the pyramid $P_n$, replacing it with two new ones and connecting them with an edge. Then we connect one of the new apexes to two adjacent vertices of the base, and connect the another new apexes to the remaining $(n-2)$ vertices of the base. As a result, the base of $W_n$ is still a $n$-gon, and its side surface consists of two quadrilaterals and $(n-2)$ triangles, as shown in Figure~\ref{fig4} for $n=6$. The polyhedron $W_n$ will be called a \emph{two-apex pyramid}. Obviously, $W_4$ is a triangular prism $\Pi_3$. In this sense, the family of two-apex pyramids $W_n$ can be considered as a generalization of the families of pyramids and prisms discussed above.
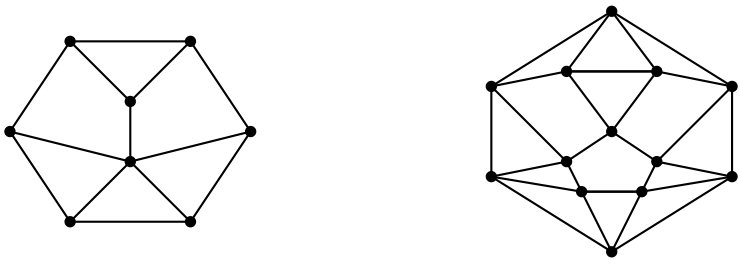
\begin{figure}[ht]
\begin{center}
\unitlength=.1mm
\begin{tikzpicture}[scale=0.4] 
\draw [line width=2pt, black] (-10,2) circle[radius=0.1cm];
\draw [line width=2pt, black] (-12,5) circle[radius=0.1cm];
\draw [line width=2pt, black] (-10,8) circle[radius=0.1cm];
\draw [line width=2pt, black] (-6,8) circle[radius=0.1cm]; 
\draw [line width=2pt, black] (-4,5) circle[radius=0.1cm]; 
\draw [line width=2pt, black] (-6,2) circle[radius=0.1cm];
\draw [line width=2pt, black] (-8,6) circle[radius=0.1cm];
\draw [line width=2pt, black] (-8,4) circle[radius=0.1cm];
\draw [thick, black] (-10,2)-- (-12,5) -- (-10,8) -- (-6,8) -- (-4,5) -- (-6,2) -- (-10,2);
\draw [thick, black] (-8,6)-- (-10,8);
\draw [thick, black] (-8,6)-- (-6,8);
\draw [thick, black] (-8,6)-- (-8,4);
\draw [thick, black] (-8,4)-- (-12,5);
\draw [thick, black] (-8,4)-- (-10,2);
\draw [thick, black] (-8,4)-- (-6,2);
\draw [thick, black] (-8,4)-- (-4,5);
\draw [line width=2pt, black] (8,5) circle[radius=0.1cm];
\draw [line width=2pt, black] (6.5,7) circle[radius=0.1cm];
\draw [line width=2pt, black] (9.5,7) circle[radius=0.1cm];
\draw [line width=2pt, black] (8,9) circle[radius=0.1cm];
\draw [line width=2pt, black] (8,1) circle[radius=0.1cm];
\draw [line width=2pt, black] (6.5,4) circle[radius=0.1cm];
\draw [line width=2pt, black] (9.5,4) circle[radius=0.1cm];
\draw [line width=2pt, black] (7,3) circle[radius=0.1cm];
\draw [line width=2pt, black] (9,3) circle[radius=0.1cm];
\draw [thick, black] (7,3)-- (6.5,4) -- (8,5) -- (9.5,4) -- (9,3) -- (7,3);
\draw [thick, black] (8,5) -- (6.5,7)-- (9.5,7) -- (8,5);
\draw [thick, black] (8,9) -- (6.5,7)-- (9.5,7) -- (8,9);
\draw [thick, black] (8,1) -- (7,3)-- (9,3) -- (8,1);
\draw [line width=2pt, black] (12,6.5) circle[radius=0.1cm];
\draw [line width=2pt, black] (4,6.5) circle[radius=0.1cm];
\draw [line width=2pt, black] (12,3.5) circle[radius=0.1cm];
\draw [line width=2pt, black] (4,3.5) circle[radius=0.1cm];
\draw [thick, black] (8,1)-- (4,3.5) -- (4,6.5) -- (8,9) -- (12,6.5) -- (12,3.5) -- (8,1);
\draw [thick, black] (9.5,7) -- (12,6.5);
\draw [thick, black] (9.5,4) -- (12,6.5);
\draw [thick, black] (9.5,4) -- (12,3.5);
\draw [thick, black] (9,3) -- (12,3.5);
\draw [thick, black] (6.5,7) -- (4,6.5);
\draw [thick, black] (6.5,4) -- (4,6.5);
\draw [thick, black] (6.5,4) -- (4,3.5);
\draw [thick, black] (7,3) -- (4,3.5);
\end{tikzpicture}
\end{center}
\caption{Two-apex pyramid $W_6$ and twisted antiprism $A(6)^*$.} \label{fig4}
\end{figure}

The medial graph for the 1-skeleton of the two-vertex pyramid $W_n$ is the 1-skeleton of the polyhedron $A(n)^*$, which was introduced in~\cite{EV20-1} and called a \emph{twisted antiprism}, see Figure~\ref{fig4} for $n=6$. For $n\geq 5$, a two-apex pyramid $W_n$ has $E = 2n+1$ edges, $V_3= n+1$ trivalent vertices, and $p_3 = n-2$ triangular faces. From the case (a) of the Theorem~\ref{theoremTrianglesBounds} we obtain the following bound: 
$$
\operatorname{vol} (W_n) \leq 2 v_{tet} \left( 2n+1 - \frac{(n-2) + (n+1) + 8}{4} \right) = 3 v_{tet} \cdot n - \frac{3}{2} v_{tet}.
$$
Since the rectification of the polyhedron $W_n$ is the polyhedron $A(n)^*$, 
$$
\operatorname{vol} (W_n) \leq \operatorname{vol} (A(n)^*)
$$
As shown in~\cite{EV20-1}, the volume of the twisted antiprism can be calculated via the volume of the antiprism the following way:
$$
\operatorname{vol}(A(n)^*) = \operatorname{vol} (A(n-1)) + \operatorname{A(3)},
$$
and $\volume(A(n)^*)$ is asymptotically equivalent to $\frac{1}{2} v_{oct}\cdot n$ as $n\to\infty$.
	
\section{Volume bound for links via the number of twists} \label{sec5}

By the \emph{volume of a hyperbolic knot or link} $K\subset S^3$ we mean the volume of a hyperbolic manifold $S^3\setminus K$.
In this Section we will establish new upper bounds for the volumes of knots and links via the combinatorial parameters of its diagram.

First of all, we recall the known bounds and illustrate these bounds for a two-bridge knot ${\bf b} (\frac{55}{17})$ as an example. 
A diagram of ${\bf b} (\frac{55}{17})$ is presented in Figure~\ref{fig5}. This figure corresponds to a continued fraction $\frac{55}{17} = 3 + \frac{1}{4+\frac{1}{4}}$ and is known as a \emph{Conway's normal form} for two-bridge knots and links~\cite{Ka96, Ro76}. Calculating the volume by the computer program \emph{SnapPy}~\cite{Snap}, we get $\operatorname{vol} (S^3\setminus {\bf b} (\frac{55}{17})) = 10.117141$.

\begin{figure}[h]
	\begin{center}
		\scalebox{0.8}{
			\begin{tikzpicture}
			\pic[
			rotate=90,
			braid/.cd,
			every strand/.style={ultra thick},
			strand 1/.style={black},
			strand 2/.style={black},
			strand 3/.style={black},
			strand 4/.style={black},]
			at (0,0) {braid={s_2^{-1} s_2^{-1} s_2^{-1} s_1 s_1 s_1 s_1 s_2^{-1} s_2^{-1} s_2^{-1} s_2^{-1}}};
			\draw [ultra thick, black] (0,3) -- (11.5,3);
			\draw[ultra thick, black] (0,1) arc (90:270:0.5);
			\draw[ultra thick, black] (0,3) arc (90:270:0.5);
			\draw[ultra thick, black] (11.5,1) arc (90:-90:0.5);
			\draw[ultra thick, black] (11.5,3) arc (90:-90:0.5);
			\end{tikzpicture}
		}
	\end{center}
	\caption{Diagram of the two-bridge knot ${\bf b}(\frac{55}{17})$.} \label{fig5}
\end{figure}
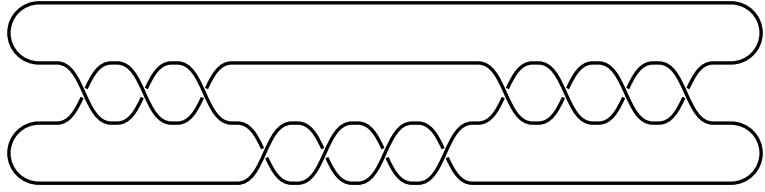

Apparently, the first known estimate of the volume of a hyperbolic knot $K$ in terms of the number of crossings $c(K)$ in its diagram was obtained in Adams dissertation~\cite{Ad83}. It was shown that if the knot $K$ is different from the figure-eight knot $4_1$, then
\begin{equation}
\volume (S^3 \setminus K) \leq v_{tet} \cdot (4 c(K) - 16). \label{eqn10}
\end{equation}
Recall that $\volume(S^3\setminus 4_1) = 2 v_{tet}$. Since $c ({\bf b}(\frac{55}{17})) =11$, the inequality (\ref{eqn10}) gives an estimate $\operatorname{vol} (S^3\setminus{\bf b} (\frac{55}{17})) \leq 28 \cdot v_{tet} = 28.418348$.

In~\cite{Ad13}, Adams improved the inequality~(\ref{eqn10}) as follows: if $c(K)\geq 5$, then
\begin{equation}
\volume{(S^3 \setminus K)} \leq v_{oct} \cdot (c(K) - 5) + 4 \cdot v_{tet}. \label{eqn11}
\end{equation}
From inequality~(\ref{eqn11}) we get $\operatorname{vol} (S^3\setminus {\bf b} (\frac{55}{17})) \leq 26.078932$.

The next family of bounds for the volumes of knots and links use the number of twists in a diagram. \emph{Twist} in the diagram $D$ of a knot or a link $K$ is the maximal chain of consecutive bigon regions, see Figure~\ref{fig200}.
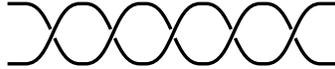
\begin{figure}[h]
	\begin{center}
		\scalebox{0.8}{
			\begin{tikzpicture}
			\pic[
			rotate=90,
			braid/.cd,
			every strand/.style={ultra thick},
			strand 1/.style={black},
			strand 2/.style={black},
			strand 3/.style={black},
			strand 4/.style={black},]
			at (0,0) {braid={s_1 s_1 s_1 s_1 s_1 }};
			\end{tikzpicture}
		}
	\end{center}
	\caption{Twist of the length five in the diagram.} \label{fig200}
\end{figure}
Equivalently, twist can be understood as a chain of several consecutive half-turns on two strands.
Moreover, all half-turns are directed in one direction: either positive or negative. The number of half-turns in the twist will be called \emph{twist length}. The number of twists in the diagram $D$ is denoted by $t(D)$. For example, for the diagram $D$ shown in Figure~\ref{fig5} we have  $t(D) = 3$.

In the appendix to Luckenby's paper~\cite{La04}, Agol and Thurston showed that the volume of any hyperbolic link $K$ can be estimated in terms of the number of twists $t(D)$ in its diagram $D$ as follows:
\begin{equation}
\volume{(S^3 \setminus K)} \leq 10 v_{tet} \cdot (t(D)-1). \label{eqn12}
\end{equation}
Moreover, this estimate is asymptotically sharp. By inequality~(\ref{eqn12}) we get the bound $\operatorname{vol} (S^3\setminus {\bf b} (\frac{55}{17})) \leq 20 \cdot v_{tet} = 20.29882$.


In~\cite{DT15} Dasbach and Tsvetkova used additional information on twists to improve the inequality obtained by  Agol and Thurston. For the diagram $D$ of the link $K$ denote by $t_i = t_i(D)$ the number of twists of length $i$ for $i\geq 1$. Note that $t(D) = \sum_{i\geq 1} t_i(D)$. Denote by $g_i = g_i(D)$ the number of twists of length at least $i$, $i \geq 1$. According to~\cite[Theorem~2.3]{DT15}, if $D$ is a reduced alternating diagram of hyperbolic alternating link $K$, then
\begin{equation}
\operatorname{vol} (S^3 \setminus K) \leq v_{tet} \cdot (4 t_1 + 6 t_2 + 8 t_3 + 10 g_4 - a), \label{eqnDT}
\end{equation}
where $a = 10$ if $g_4\neq 0$, $a=7$ if $t_3 \neq 0$, $a = 6$ otherwise. Later in~\cite{DT19} Dasbach and Tsvietkova proved that the bound~(\ref{eqnDT}) is also true for the case when the diagram is not alternating.

Adams~\cite[Theorem~3.1]{Ad17-1} improved the result obtained by Dasbach and Tsvietkova as follows.
Let $K$ be a hyperbolic link admitting a reduced alternating diagram $D$ with $c(D)\geq 5$ and $t(D)\geq 3$. Moreover, we assume that $K$ is not the Borromean rings $6^3_2$. Then
\begin{equation}
\volume{(S^3 \setminus K)} < t_1 \cdot v_{oct} + t_2 \cdot 6v_{tet} + t_3 \cdot 16 \Lambda \left(\frac{\pi}{8} \right) + t_4 \cdot 20 \Lambda \left(\frac{\pi}{10} \right) + g_5 \cdot 10v_{tet} - a, \label{eqn13}
\end{equation}
where
\begin{equation}
a = \begin{cases}
7 v_{oct} - 10 v_{tet}, & \text{ if } g_2 =0, \\
11 v_{tet}, & \text{ if } g_3 = 0\text{ and } t_2 \geq 1, \\
32 \Lambda\left(\frac{\pi}{8}\right) + 5 v_{tet} - v_{oct} - 14 \Lambda\left(\frac{\pi}{7} \right) , & \text{ if } g_4 = 0 \text{ and } t_3 \geq 1, \\
40 \Lambda\left(\frac{\pi}{10}\right) + 12 \Lambda\left(\frac{\pi}{6}\right) - 2 v_{tet} - 8 \Lambda\left(\frac{\pi}{4}\right) - 18 \Lambda \left(\frac{\pi}{9}\right), & \text{ if } g_5 = 0 \text{ and } t_4 \geq 1, \\
4 v_{tet} + 12\Lambda\left(\frac{\pi}{6}\right) + 60\Lambda\left(\frac{\pi}{10}\right) - 54\Lambda\left(\frac{\pi}{9} \right) , & \text{ if } g_5 \geq 1.
\end{cases} \label{eqn14}
\end{equation}

Calculating the values of the Lobachevsky function specified in~(\ref{eqn13}) and~(\ref{eqn14}) with an accuracy of up to six digits, we get the inequality
\begin{equation}
\volume{(S^3 \setminus K)} < 3.663863 \cdot t_1 + 6.089646 \cdot t_2 + 7.854977 \cdot t_3 + 9.237551 \cdot t_4 + 10.149416 \cdot g_5 - a, \label{eqn15}
\end{equation}
where $a$ takes the following values:
\begin{equation}
a = \begin{cases}
15.497263, & \text{ if } g_2 =0, \\
11.164351, & \text{ if } g_3 = 0 \text{ and } t_2 \geq 1, \\
10.088228, & \text{ if } g_4 = 0 \text{ and } t_3 \geq 1, \\
10.287338, & \text{ if } g_5 = 0 \text{ and } t_4 \geq 1, \\
12.111063, & \text{ if } g_5 \geq 1.
\end{cases} \label{eqn16}
\end{equation}
Note that the formulae (\ref{eqn13})--(\ref{eqn16}) give the bound $\operatorname{vol} (S^3\setminus {\bf b}(\frac{55}{17})) < 16.0426$.

\smallskip

The following bounds for the volume of an alternating knot in terms of the coefficients of its Jones polynomial were obtained by Dasbach and Lin~\cite{DL07}. Let $K$ be a simple alternating knot that is not toric. Let its Jones polynomial have the form
$$
V_K (t) = a_n t^n + a_{n+1} t^{n+1} \ldots + a_{m-1} t^{m-1} + a_m t^m.
$$
Then
\begin{equation}
v_{oct} (\max(|a_{m-1}, |a_{n+1}| - 1)) \leq \operatorname{vol} (S^3 \setminus K) \leq 10 v_3 (|a_{n+1}| + |a_{m-1}| - 1).
\end{equation}

According to~\cite{LM}, Jones polynomial for ${\bf b}(\frac{55}{17})$ is equal to 
$$
V (t) = t^3-t^4+ 3t^5-5t^6+ 7t^7-8t^8+ 9t^9-8t^{10}+ 6t^{11}-4t^{12}+ 2t^{13}-t^{14}.
$$
Hence, the following bounds hold: 
$$
2 \cdot v_{oct}\leq\operatorname{vol} \left(S^3 \setminus {\bf b} \left(\frac{55}{17}\right) \right) \leq 20 \cdot v_3.
$$

Knot ${\bf b} \left(\frac{55}{17}\right)$ belongs to the class of two-bridge knots and links. Recall~\cite{Ka96}, that for coprime integers $p$ and $q$ such that $p\geq 2$ and $0<q<p$, the two-bridge link ${\bf b}(\frac{p}{q})$ is defined. If $p$ is odd, then ${\bf b}(\frac{p}{q})$ is a knot. If $p$ is even, then ${\bf b}(\frac{p}{q})$ is a 2-component link. The diagram of the link ${\bf b}(\frac{p}{q})$ is determined by the decomposition of the rational number $\frac{p}{q}$ into a continued fraction. Namely, if $\frac{p}{q}= [a_1, \ldots, a_n]$, where
$$
[a_1,a_2,\ldots,a_{n-1},a_n]=a_1 +
\frac{\displaystyle 1}{\displaystyle a_2 + \, \cdots\, +
	\frac{1}{a_{n-1}+\frac{1}{a_{n}}}},
$$
then ${\bf b}(\frac{p}{q})$ has a diagram as in the Figure~\ref{fig12} depending on whether $n$ is even or odd, where $a_1, \ldots, a_n$ indicate the number of half-turns on two strands. Diagrams of two-bridge knots (odd $n$) and links (even $n$) shown in Figure~\ref{fig12} are called Conway's normal forms.
\begin{figure}[h]
\begin{center} 
\scalebox{0.7}{
\begin{tikzpicture} 
 \pic[
  rotate=90,
  braid/.cd,
  every strand/.style={ultra thick},
  strand 1/.style={black},
  strand 2/.style={black},
  strand 3/.style={black},] 
at (0,0) {braid={s_2^{-1}}}; 
\node (A) at (2.,1.8) {$a_1$};
\node (A) at (2.,1.5) {$\dots$};
 \pic[
  rotate=90,
  braid/.cd,
  every strand/.style={ultra thick},
  strand 1/.style={black},
  strand 2/.style={black},
  strand 3/.style={black},] 
at (2.5,0) {braid={s_2^{-1}}}; 
\draw  [ultra thick, black]  (1.5,0) -- (2.5,0);
 \pic[
  rotate=90,
  braid/.cd,
  every strand/.style={ultra thick},
  strand 1/.style={black},
  strand 2/.style={black},
  strand 3/.style={black},] 
at (4,0) {braid={s_1}}; 
\node (A) at (6.,0.8) {$a_2$};
\node (A) at (6.,0.5) {$\dots$};
 \pic[
  rotate=90,
  braid/.cd,
  every strand/.style={ultra thick},
  strand 1/.style={black},
  strand 2/.style={black},
  strand 3/.style={black},] 
at (6.5,0) {braid={s_1}}; 
\draw  [ultra thick, black]  (4,2) -- (8,2);
 \pic[
  rotate=90,
  braid/.cd,
  every strand/.style={ultra thick},
  strand 1/.style={black},
  strand 2/.style={black},
  strand 3/.style={black},] 
at (8,0) {braid={s_2^{-1}}}; 
\node (A) at (10.5,0) {$\dots$};
\node (A) at (10.5,1) {$\dots$};
\node (A) at (10.5,2) {$\dots$};
 \pic[
  rotate=90,
  braid/.cd,
  every strand/.style={ultra thick},
  strand 1/.style={black},
  strand 2/.style={black},
  strand 3/.style={black},] 
at (11.5,0) {braid={s_2^{-1}}}; 
\node (A) at (13.5,1.8) {$a_n$};
\node (A) at (13.5,1.5) {$\dots$};
\draw  [ultra thick, black]  (13,0) -- (14,0);
 \pic[
  rotate=90,
  braid/.cd,
  every strand/.style={ultra thick},
  strand 1/.style={black},
  strand 2/.style={black},
  strand 3/.style={black},] 
at (14,0) {braid={s_2^{-1}}}; 
\draw  [ultra thick, black]  (0,3) -- (15.5,3);
 \draw[ultra thick, black] (0,1) arc (90:270:0.5);
  \draw[ultra thick, black] (0,3) arc (90:270:0.5);
   \draw[ultra thick, black] (15.5,1) arc (90:-90:0.5);
    \draw[ultra thick, black] (15.5,3) arc (90:-90:0.5);
 \pic[
  rotate=90,
  braid/.cd,
  every strand/.style={ultra thick},
  strand 1/.style={black},
  strand 2/.style={black},
  strand 3/.style={black},] 
at (0,4) {braid={s_2^{-1}}}; 
\node (A) at (2.,5.8) {$a_1$};
\node (A) at (2.,5.5) {$\dots$};
 \pic[
  rotate=90,
  braid/.cd,
  every strand/.style={ultra thick},
  strand 1/.style={black},
  strand 2/.style={black},
  strand 3/.style={black},] 
at (2.5,4) {braid={s_2^{-1}}}; 
\draw  [ultra thick, black]  (1.5,4) -- (2.5,4);
 \pic[
  rotate=90,
  braid/.cd,
  every strand/.style={ultra thick},
  strand 1/.style={black},
  strand 2/.style={black},
  strand 3/.style={black},] 
at (4,4) {braid={s_1}}; 
\node (A) at (6.,4.8) {$a_2$};
\node (A) at (6.,4.5) {$\dots$};
 \pic[
  rotate=90,
  braid/.cd,
  every strand/.style={ultra thick},
  strand 1/.style={black},
  strand 2/.style={black},
  strand 3/.style={black},] 
at (6.5,4) {braid={s_1}}; 
\draw  [ultra thick, black]  (4,6) -- (8,6);
 \pic[
  rotate=90,
  braid/.cd,
  every strand/.style={ultra thick},
  strand 1/.style={black},
  strand 2/.style={black},
  strand 3/.style={black},] 
at (8,4) {braid={s_2^{-1}}}; 
\node (A) at (10.5,4) {$\dots$};
\node (A) at (10.5,5) {$\dots$};
\node (A) at (10.5,6) {$\dots$};
 \pic[
  rotate=90,
  braid/.cd,
  every strand/.style={ultra thick},
  strand 1/.style={black},
  strand 2/.style={black},
  strand 3/.style={black},] 
at (11.5,4) {braid={s_1}}; 
\node (A) at (13.5,4.8) {$a_n$};
\node (A) at (13.5,4.5) {$\dots$};
\draw  [ultra thick, black]  (13,0) -- (14,0);
 \pic[
  rotate=90,
  braid/.cd,
  every strand/.style={ultra thick},
  strand 1/.style={black},
  strand 2/.style={black},
  strand 3/.style={black},] 
at (14,4) {braid={s_1}}; 
\draw  [ultra thick, black]  (11.5,6) -- (15.5,6);
\draw  [ultra thick, black]  (0,7) -- (15.5,7);
 \draw[ultra thick, black] (0,5) arc (90:270:0.5);
  \draw[ultra thick, black] (0,7) arc (90:270:0.5);
   \draw[ultra thick, black] (15.5,6) arc (90:-90:0.5);
    \draw[ultra thick, black] (15.5,7) arc (90:-90:1.5);
\end{tikzpicture}
}
\end{center}
\caption{Conway's normal forms for 2-bridge links and knots.}\label{fig12}
\end{figure}
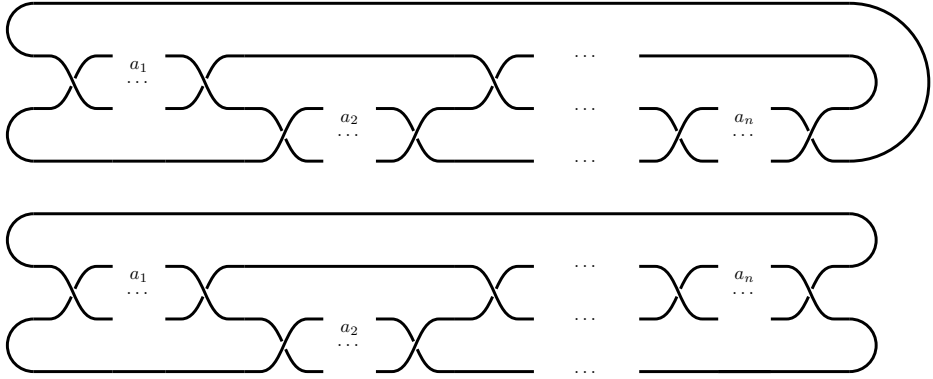

For the volumes of two-bridge links the upper and lower bounds were obtained in~\cite{GF06}. Let $D$ be a reduced alternating diagram of a two-bridge link $K$, then
$$
2 v_{tet}\cdot t(D)-2.7066\leq\operatorname{vol} (S^3\setminus K) \leq 2 v_{oct}\cdot(t(D)-1).
$$
The obtained bounds were used in~\cite{PV09} to estimate Matveev's complexity of hyperbolic 3-manifolds represented as cyclic branched covers of 2-bridge knots and links. 
The proof of the upper bound is based on the fact that the full augmentation without half-turns (see a definition below) of a two-bridge link with $t(D)$ twists is the belt sum of $t(D)$ copies of Borromean rings.
Applying the last bound to ${\bf b}(\frac{55}{17})$, we get:
$$
6 \cdot v_{tet} - 2.7066 \leq\operatorname{vol} \left(S^3 \setminus {\bf b} \left(\frac{55}{17}\right) \right) \leq 4 \cdot v_{oct}.
$$

\medskip

Now we go back to considering arbitrary hyperbolic knots and links. 
In the Theorem~\ref{theoremKnotsBounds} we obtain an inequality that improves the bounds (\ref{eqn12}) and (\ref{eqn13}) for the cases when the number of twists in the diagram is large enough.

\smallskip

\begin{theorem} \label{theoremKnotsBounds}
	Let $D$ be a hyperbolic diagram of a link $K$ with $t(D)$ twists. If $t(D)>$8, then
	\begin{equation}
	\volume{(S^3 \setminus K)} \leq 10 v_{tet} \cdot (t(D)-1.4). \label{eqn17}
	\end{equation}
\end{theorem}

\begin{proof}
	Starting from a diagram of $K$ we will construct a link $L$ such that $\operatorname{vol} (S^3\setminus K) <\operatorname{vol} (S^3\setminus L)$ and bound the volume of the manifold $S^3\setminus L$ by splitting it into two ideal right-angled polyhedra.
	
	\underline{Step 1.} 
	Let $D$ be a diagram of a link $K$ with $t = t(D)$ twists. Denote the lengths of the twists by $n_1, n_2, \ldots, n_t$. Similarly to~\cite{La04} and~\cite[Chapter~7]{Pu11}, we construct a new link from the diagram $D$ as follows. If $| n_i | \geq 2$ then replace the maximum possible number of full turns $\lfloor\frac{|n_i|}{2}\rfloor$ in the $i$-th twist, $i=1, \ldots, t$, with a new link component that covers this twist. If $|n_i| = 1$ then we just add a new link component that cover the twist. The resulting link $J$ is called \emph{full augmentation} of the link $K$ (see, for example,~\cite{Pu11}). Thus, $J$ has $t(D)$ more components than the original link $K$. The new components in $J$ will be further referred to as \emph{vertical}. Note that the initial link $K$ can be obtained by Dehn's surgeries  on vertical components of $J$. For example, the knot diagram ${\bf b} (\frac{55}{17})$, shown in Figure~\ref{fig5}, has twists of length $3$, $4$ and $4$.
	The corresponding link $J_4$ has $4$ components and shown in the Figure~\ref{fig6}.
	\begin{figure}[h]
		\begin{center}
			\scalebox{0.8}{
				\begin{tikzpicture}
				\pic[
				rotate=90,
				braid/.cd,
				every strand/.style={ultra thick},
				strand 1/.style={black},
				strand 2/.style={black},
				strand 3/.style={black},
				strand 4/.style={black},]
				at (0,0) {braid={s_2^{-1}}};
				\draw [ultra thick, black] (0,3) -- (6,3);
				\draw[ultra thick, black] (0,1) arc (90:270:0.5);
				\draw[ultra thick, black] (0,3) arc (90:270:0.5);
				\draw[ultra thick, black] (6,1) arc (90:-90:0.5);
				\draw[ultra thick, black] (6,3) arc (90:-90:0.5);
				\draw[ultra thick, red] (2,2.25) arc (45:315: 0.2cm and 1cm);
				\draw [ultra thick, red] (2,1.2) -- (2,1.8);
				\draw[ultra thick, red] (5,2.25) arc (45:315: 0.2cm and 1cm);
				\draw [ultra thick, red] (5,1.2) -- (5,1.8);
				\draw[ultra thick, red] (3.5,1.25) arc (45:315: 0.2cm and 1cm);
				\draw [ultra thick, red] (3.5,0.2) -- (3.5,0.8);
				\draw [ultra thick, black] (1,0) -- (3,0);
				\draw [ultra thick, black] (1.8,1) -- (3,1);
				\draw [ultra thick, black] (3.3,1) -- (4.5,1);
				\draw [ultra thick, black] (1.8,2) -- (4.5,2);
				\draw [ultra thick, black] (4.8,1) -- (6,1);
				\draw [ultra thick, black] (4.8,2) -- (6,2);
				\draw [ultra thick, black] (3.3,0) -- (6,0);
				\end{tikzpicture}
			}
		\end{center}
		\caption{Four-component link $J_4$.} \label{fig6}
	\end{figure}
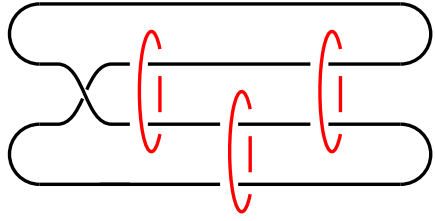
Three of the four components of the link $J_4$ are vertical, and surgeries on these components give the knot ${\bf b}(\frac{55}{17})$. Hence, $\operatorname{vol} (S^3\setminus {\bf b}(\frac{55}{17})) < \operatorname{vol} (S^3 \setminus J_4)$.
	
\underline{Step 2.} If the twist in $D$ had an odd length, then one half-turn will remain from it in the diagram of a link $J$. Let us  change the link $J$ to the link $L$ which does not have such half-turns. To do this, we apply \emph{Adams transformation}, shown in the Figure~\ref{fig100}, the required number of times. The resulting link $L$ is called \emph{full augmentation without half-turns} of the link $K$ (see, for example, terminology from~\cite{Kw20}).
\begin{figure}[h]
\begin{center} 
\scalebox{1.0}{
\begin{tikzpicture} 
 \draw  [ultra thick, black]  (6.5,2) -- (7.5,2);
\draw  [ultra thick, black]  (6.5,1) -- (7.5,1);
\draw  [ultra thick, black]  (7.8,2) -- (9,2);
\draw  [ultra thick, black]  (7.8,1) -- (9,1);
\draw[ultra thick, red]  (8,2.1)  arc (50:320: 0.2cm and 0.9cm);
\draw  [ultra thick, red]  (8,1.2) -- (8,1.8);
\node  at (7.5,0) {(2)};
\pic[
 rotate=90,
 braid/.cd,
 every strand/.style={ultra thick},
 strand 1/.style={black},
 strand 2/.style={black}] 
at (0,1) {braid={s_1^{-1} }}; 
\draw  [ultra thick, black]  (1.8,2) -- (3,2);
\draw  [ultra thick, black]  (1.8,1) -- (3,1);
\draw[ultra thick, red]  (2,2.1)  arc (50:320: 0.2cm and 0.9cm);
\draw  [ultra thick, red]  (2,1.2) -- (2,1.8);
\node  at (1.5,0) {(a)};
 \end{tikzpicture}
}
\end{center}
\caption{Adams transformation.} \label{fig100}
\end{figure}
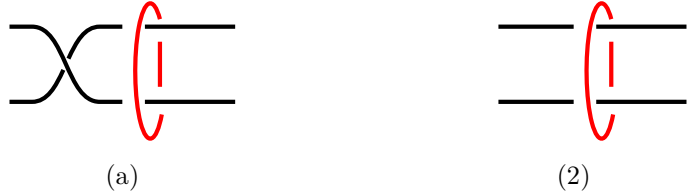
As shown by Adams~\cite[Corollary~5.1]{Ad85}, the application of this transformation preserves the property of link to be hyperbolic and does not change the volume, hence $\operatorname{vol}(S^3\setminus J) = \operatorname{vol}(S^3\setminus L)$. Note that links $J$ and $L$ have the same number of vertical components and $L$ has no half-turns in the diagram. For example, the link $L_5$ obtained by Adams transformation from the link $J_4$ is shown in Figure~\ref{fig7}, it has $5$ components, three of which are vertical.
\begin{figure}[h]
\begin{center} 
\scalebox{0.8}{
\begin{tikzpicture} 
\draw  [ultra thick, black]  (0,3) -- (5,3);
 \draw[ultra thick, black] (0,1) arc (90:270:0.5);
  \draw[ultra thick, black] (0,3) arc (90:270:0.5);
   \draw[ultra thick, black] (5,1) arc (90:-90:0.5);
    \draw[ultra thick, black] (5,3) arc (90:-90:0.5);
\draw[ultra thick, red]  (1,2.25)  arc (45:315: 0.2cm and 1cm);
\draw  [ultra thick, red]  (1,1.2) -- (1,1.8);
\draw  [ultra thick, black]  (0,1) -- (0.5,1);
\draw  [ultra thick, black]  (0,2) -- (0.5,2);
\draw  [ultra thick, black]  (0,0) -- (2,0);
\draw  [ultra thick, black]  (0.8,1) -- (2,1);
\draw  [ultra thick, black]  (2.3,1) -- (3.5,1);
\draw  [ultra thick, black]  (0.8,2) -- (3.5,2);
\draw[ultra thick, red]  (4,2.25)  arc (45:315: 0.2cm and 1cm);
 \draw  [ultra thick, red]  (4,1.2) -- (4,1.8);
%
\draw[ultra thick, red]  (2.5,1.25)  arc (45:315: 0.2cm and 1cm);
 \draw  [ultra thick, red]  (2.5,0.2) -- (2.5,0.8);
 \draw  [ultra thick, black]  (3.8,1) -- (5,1);
 \draw  [ultra thick, black]  (3.8,2) -- (5,2);
 \draw  [ultra thick, black]  (2.3,0) -- (5,0);
\end{tikzpicture}
}
\end{center}
\caption{Five-component link $L_5$.} \label{fig7}
\end{figure}
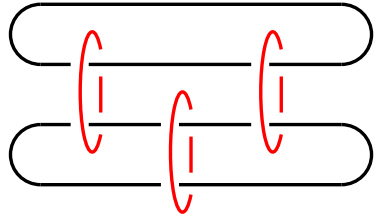

\underline{Step 3.} Applying the method from~\cite{La04}, we decompose  $S^3\setminus L$ to the union of two copies of an ideal right-angled hyperbolic polyhedron. Firstly , we replace each vertical component of the link $L$ with a pair of triangles with a common vertex as shown in Figure~\ref{fig8} for the link $L_5$. In this case, the edges corresponding to the triangles we will call red, and the edges connecting the triangles we will call black.
\begin{figure}[h]
\begin{center} 
\scalebox{1.0}{
\begin{tikzpicture} 
\draw  [ultra thick, black]  (0,3) -- (4,3);
 \draw[ultra thick, black] (0,1) arc (90:270:0.5);
  \draw[ultra thick, black] (0,3) arc (90:270:0.5);
   \draw[ultra thick, black] (4,1) arc (90:-90:0.5);
    \draw[ultra thick, black] (4,3) arc (90:-90:0.5);
\filldraw  [ultra thick, cyan!30]  (0.5,1.5) -- (0,1) -- (0,2) -- (0.5,1.5);
\filldraw  [ultra thick, cyan!30]  (0.5,1.5) -- (1,1) -- (1,2) -- (0.5,1.5);
\draw  [ultra thick, red]  (0,1) -- (0.5,1.5);
\draw  [ultra thick, red]  (0,2) -- (0.5,1.5);
\draw  [ultra thick, red]  (0,1) -- (0,2);
\draw  [ultra thick, red]  (1,1) -- (0.5,1.5);
\draw  [ultra thick, red]  (1,2) -- (0.5,1.5);
\draw  [ultra thick, red]  (1,1) -- (1,2);
\draw  [ultra thick, black]  (0,0) -- (1.5,0);
\draw  [ultra thick, black]  (1,1) -- (1.5,1);
\draw  [ultra thick, black]  (2.5,1) -- (3,1);
\draw  [ultra thick, black]  (1,2) -- (3,2);
\filldraw [line width=2pt, red] (0.5,1.5) circle[radius=0.06cm];
%
\filldraw  [ultra thick, cyan!30]  (3.5,1.5) -- (3,1) -- (3,2) -- (3.5,1.5);
\filldraw  [ultra thick, cyan!30]  (3.5,1.5) -- (4,1) -- (4,2) -- (3.5,1.5);
\draw  [ultra thick, red]  (3,1) -- (3.5,1.5);
\draw  [ultra thick, red]  (3,2) -- (3.5,1.5);
\draw  [ultra thick, red]  (3,1) -- (3,2);
\draw  [ultra thick, red]  (4,1) -- (3.5,1.5);
\draw  [ultra thick, red]  (4,2) -- (3.5,1.5);
\draw  [ultra thick, red]  (4,1) -- (4,2);
\filldraw [line width=2pt, red] (3.5,1.5) circle[radius=0.06cm];
%
\filldraw  [ultra thick, cyan!30]  (2,0.5) -- (1.5,0) -- (1.5,1) -- (2,0.5);
\filldraw  [ultra thick, cyan!30]  (2,0.5) -- (2.5,0) -- (2.5,1) -- (2,0.5);
\draw  [ultra thick, red]  (1.5,0) -- (2,0.5);
\draw  [ultra thick, red]  (1.5,1) -- (2,0.5);
\draw  [ultra thick, red]  (1.5,0) -- (1.5,1);
\draw  [ultra thick, red]  (2.5,0) -- (2,0.5);
\draw  [ultra thick, red]  (2.5,1) -- (2,0.5);
\draw  [ultra thick, red]  (2.5,0) -- (2.5,1);
\filldraw [line width=2pt, red] (2,0.5) circle[radius=0.06cm];
%
 \draw  [ultra thick, black]  (2.5,0) -- (4,0);
\end{tikzpicture}
}
\end{center}
\caption{Replacing vertical components of the link $L_5$ with pairs of triangles.} \label{fig8}
\end{figure}
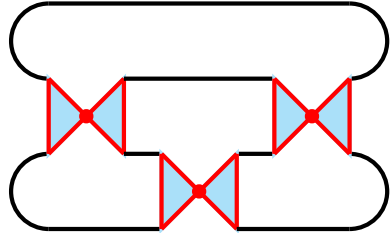
Secondly, we contract each black edge into a point and denote the resulting polyhedron by $P$. Since each black edge was incident to two trivalent vertices, after the black edges are contracted, new quadrivalent vertices will appear. We will call them by \emph{black}. As a result, all vertices of the polyhedron $P$ are quadrivalent, moreover, $t(D)$ of them are red and $2 t(D)$ are black. Thus, $P$ has $V = 3 t(D)$ vertices and at least $2 t(D)$ triangular faces, which, with a two-color chessboard coloring of the faces of the polyhedron, will turn out to be colored in the same color. As noted in~\cite{La04}, the polyhedron $P$ is an ideal right-angled polyhedron and $\operatorname{vol}(S^3\setminus L) = 2\operatorname{vol}(P)$. The polyhedron $P_5$ corresponding to the diagram of the link $L_5$ is shown in the Figure~\ref{fig9}.
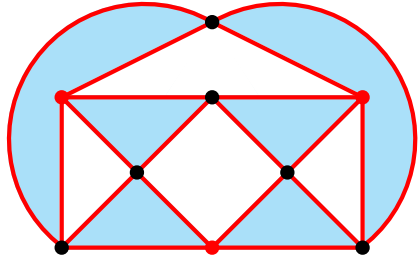
\begin{figure}[h]
\begin{center} 
\unitlength=.1mm
\scalebox{1.0}{
\begin{tikzpicture}[scale=1.0]
  \filldraw  [ultra thick, cyan!30]  (2,3) arc (60:235:1.8);
  \filldraw  [ultra thick, white]  (0,0) -- (0,2) -- (2,3) -- (0,0);
    \filldraw  [thick, cyan!30]  (2,3) arc (120:-55:1.8);
  \filldraw  [ultra thick, white]  (4,0) -- (4,2) -- (2,3) -- (4,0);
  \filldraw  [ultra thick, white]  (0,0) -- (0,2) -- (1,1) -- (0,0); 
    \filldraw  [ultra thick, white]  (4,0) -- (4,2) -- (3,1) -- (4,0);
%
\filldraw  [ultra thick, cyan!30]  (2,0) -- (0,0) -- (1,1) -- (2,0);
\filldraw  [ultra thick, cyan!30]  (2,0) -- (4,0) -- (3,1) -- (2,0);
\filldraw  [ultra thick, cyan!30]  (2,2) -- (0,2) -- (1,1) -- (2,2);
\filldraw  [ultra thick, cyan!30]  (2,2) -- (4,2) -- (3,1) -- (2,2);
 \draw[ultra thick, red] (2,3) arc (60:235:1.8);
  \draw[ultra thick, red] (2,3) arc (120:-55:1.8);
\draw  [ultra thick, red]  (1,1) -- (0,0);
\draw  [ultra thick, red]  (1,1) -- (0,2);
\draw  [ultra thick, red]  (1,1) -- (2,2);
\draw  [ultra thick, red]  (1,1) -- (2,0);
\draw  [ultra thick, red]  (3,1) -- (2,0);
\draw  [ultra thick, red]  (3,1) -- (2,2);
\draw  [ultra thick, red]  (3,1) -- (4,2);
\draw  [ultra thick, red]  (3,1) -- (4,0);
\draw  [ultra thick, red]  (0,0) -- (4,0);
\draw  [ultra thick, red]  (0,2) -- (4,2);
\draw  [ultra thick, red]  (0,0) -- (0,2);
\draw  [ultra thick, red]  (4,0) -- (4,2);
\draw  [ultra thick, red]  (2,3) -- (0,2); 
\draw  [ultra thick, red]  (2,3) -- (4,2);
\filldraw [line width=2pt, black] (0,0) circle[radius=0.06cm];
\filldraw [line width=2pt, red] (2,0) circle[radius=0.06cm];
\filldraw [line width=2pt, black] (4,0) circle[radius=0.06cm];
\filldraw [line width=2pt, black] (1,1) circle[radius=0.06cm];
\filldraw [line width=2pt, black] (3,1) circle[radius=0.06cm];
\filldraw [line width=2pt, red] (0,2) circle[radius=0.06cm];
\filldraw [line width=2pt, black] (2,2) circle[radius=0.06cm];
\filldraw [line width=2pt, red] (4,2) circle[radius=0.06cm];
\filldraw [line width=2pt, black] (2,3) circle[radius=0.06cm];
\end{tikzpicture}
}
\end{center}
\caption{Polyhedron $P_5$ for the link $L_5$.} \label{fig9}
\end{figure}

As in the proof of Lemma~\ref{lemma4.1}, let us consider the union $DP$ of bipyramids obtained by doubling pyramids with a common apex $v$, splitting $P$. The volume of $DP$ is twice the volume of $P$ and therefore coincides with the volume of $\volume{(S^3\setminus L)}$. For $n\geq 4$ each $n$-bipyramid is a union of $n$ ideal tetrahedra. Each triangular bipyramid, corresponding to one of the $2t(D)$ triangular faces of the same color, will be divided into two tetrahedra. The remaining triangular bipyramids will be considered divided into $3$ tetrahedra. Since the number of vertices $V=3t(D)>24$, then (see proof of item (b) of Lemma~\ref{lemma4.2}) the apex $v$ can be chosen so that the sum of the sizes of the four faces adjacent to $v$ is at least $16$. We get the estimate
$$
\volume(DP) \leqslant 4 v_{tet} \cdot \left(V - \frac{2t(D)+k}{4} \right),
$$ 
where $k$ is the number of tetrahedra corresponding to degenerate bipyramids. Let us estimate the value of $k$. The vertex $v$ is adjacent to two triangles that are colored in the same color in the two-color chessboard coloring of the faces of the polyhedron $P$. The two bipyramids that have these triangles as bases correspond to $4$ degenerate tetrahedra. The other two faces adjacent to $v$ can be of any size, but their total number of sides is not less than $10$. So two bipyramids that have these two faces as bases correspond to at least $10$ degenerate tetrahedra. Hence we get that $k\geq 4+10=14$, therefore
$$
\volume{(S^3 \setminus K)} < \volume{(S^3 \setminus L)} = \volume(DP) \leqslant 4 v_{tet} \cdot \left( 3t(D) - \frac{2t(D)+14}{4}\right),
$$
which completes the proof of the Theorem.
\end{proof}

\smallskip

\begin{remark}
{\rm 
The paper~\cite{DT15} gives a bounds that takes into account the number of triangles $\Delta$ of the polyhedron $P$ constructed from the link $K$:
$$
\operatorname{vol} (S^3 \setminus K) \leq  v_{tet}  \cdot (4 t_1 + 6 t_2 + 8 t_3 + 10 g_4 - a - \Delta ).
$$
One can refine the formula (\ref{eqn17}) in a similar way. Indeed, if the polyhedron $P$ constructed from the link $K$ have $\Delta+2t(D)$ triangles, then
\begin{equation}
\volume{(S^3 \setminus K)}  \leq 10 v_{tet} \cdot \left( t(D)-1.3- \frac{\Delta}{10} \right).  \label{eqn18}
\end{equation}
}
\end{remark}

\smallskip 

\begin{remark}
{\rm If link $K$ have a diagram $D$ with $t(D)>8$ twists, then the bound (\ref{eqn18}) improves the bound (\ref{eqn12}).
}
\end{remark}

\smallskip 

\begin{remark}
{\rm 
Assume that link $K$ has such a reduced alternating diagram $D$ that $t(D) > 8$ and all twists have a length of at least $5$, that is, $t_1 = t_2 = t_3 = t_4 = 0$. Then the bound (\ref{eqn18}) improves the bound (\ref{eqn13}). In fact, comparing these estimates, we get
$$
10 v_{tet} (t(D) - 1.4) < 10 v_{tet} \, t(D) -12.111063,
$$
since $14 v_{tet} = 14.209174$.
}
\end{remark}

Note that apart from the upper bounds for the volumes of hyperbolic links in terms the number of twists, there are few lower bounds. As an example, let us consider the bound obtained in~\cite{FKP}. Suppose the link $K$ have such a simple reduced alternating diagram $D$ that $t(D)\geq 2$ and all twists have a length of at least $7$, that is, $t_1 = t_2 = t_3 = t_4 = t_5=t_6 = 0$. Then
\begin{equation}
0.70735 \cdot (t(D)-1) < \volume{(S^3 \setminus K)}.  \label{eqn19}
\end{equation}

\smallskip

\begin{remark}
{\rm 
The bound from the Theorem~\ref{theoremKnotsBounds} can be clarified if there is an information about which faces, other than $2 t(D)$ triangular, are present in the polyhedron $P$ constructed from the diagram $D$. Recall that with a two-color chess coloring, $2t(D)$ triangles have one color, we call it \emph{dark}, and the other polygons have a different color, we call it \emph{white}. Denote by $f_n$, $n\geq 3$, the number of white $n$-gons in $P$. For example, for the polyhedron $P$ shown in Figure~\ref{fig9}, we have $f_3= 3$, $f_4=2$, $f_n=0$, $n\geq 5$.
}
\end{remark}

\smallskip

\begin{corollary}
 	Let $D$ be the diagram of the hyperbolic link $K$. Let $f_n$ be the number of white $n$-gonal faces in the ideal right-angled polyhedron $P$ constructed from a full augmentation without half-turns of the link $K$. Then
 	$$
 	\operatorname{vol} (S^3 \setminus K) \leq (4 t(D) - 8) v_{tet} + 2 \sum_{n \geq 3} n f_n \Lambda \left( \frac{\pi}{n} \right).
 	$$
\end{corollary}
 
\begin{proof}
		As in the proof of Lemma~\ref{lemma4.1}, we use the fact that the volume of $n$-bipyramid is at most $2 n\Lambda (\frac{\pi}{n})$.
\end{proof}
	


\begin{thebibliography}{00} 

\bibitem{Ad83}
C.~Adams, \textit{Hyperbolic Structures on Link Complements},
Ph.D. thesis, University of Wisconsin, 1983.	

\bibitem{Ad85}
C.~Adams, \textit{Thrice-punctured spheres in hyperbolic 3-manifolds},
Trans. Amer. Math. Soc.,  \textbf{287} (1985), 645--656.		


\bibitem{Ad13}
C.~Adams, \textit{Triple crossing number of knots and links},
J. Knot Theory Ramifications,  \textbf{22:2} (2013),  paper number 1350006.

\bibitem{Ad17-1}
C.~Adams, \textit{Bipyramids and bounds on volumes of hyperbolic links},
Topology and its Applications,  \textbf{222} (2017), 100--114.		
	
		
\bibitem{ABEV}
S.~Alexandrov, N.~Bogachev, A.~Egorov, A.~Vesnin, \textit{On volumes of hyperbolic right-angled polyhedra}, Sbornik Math.    \textbf{214:2} (2023), 3--22. 
(in Russian). English preprint version is available at \url{https://arxiv.org/abs/2111.08789}. 

\bibitem{An70a}
E.\,M.~Andreev, \textit{On convex polyhedra in Lobacevskii spaces},  Math. USSR-Sb., \textbf{10:3} (1970), 413--440.

\bibitem{An70b}
E.M.~Andreev, \emph{On convex polyhedra of finite volume in Lobachevskii space,} Math. USSR-Sb. \textbf{12:2} (1970), 255--259. 

\bibitem{At09}
C.~Atkinson, \textit{Volume estimates for equiangular hyperbolic Coxeter polyhedra,} Algebraic \& Geometric Topology \textbf{9} (2009), 1225--1254.
		
\bibitem{At11}
C.~Atkinson, \textit{Two-sided combinatorial volume bounds for non-obtuse hyperbolic polyhedra,} Geomeriae Dedicata \textbf{153} (2011), 177--211.
		
\bibitem{BB02}
X.~Bao, F.~Bonahon, \textit{Hyperideal polyhedra in hyperbolic 3-space}, Bull. Soc. Math. France, \textbf{130:3}, (2002), 457--491,
		
\bibitem{Be21}
G.~Belletti, \textit{The maximum volume of hyperbolic polyhedra}, Trans. Amer. Math. Soc., \textbf{374} (2021), 1125--1153.  		
		
\bibitem{Br05}
G.~Brinkmann, S.~Greenberg, C.~Greenhill, B.\,D.~McKay, R.~Thomas, P.~Wollan, \textit{Generation of simple quadrangulations of the sphere}, Discrete Mathematics \textbf{305} (2005), 33--54.
		
		
\bibitem{CK99}
Y.~Cho, H.~Kim, \textit{On the volume formula for hyperbolic tetrahedra}, Discrete  Comput. Geom. \textbf{22} (1999), 347--366.
		
\bibitem{Co34}
H.\,S.\,M.~Coxeter, \textit{Discrete groups generated by reflections}, Ann. Math., \textbf{35:2} (1934), 588-621.
		
\bibitem{DL07}
O.\,T.~Dasbach, X.-S.~Lin, \textit{A volumish theorem for the Jones polynomial of alternating knots}, Pacific Journal of Mathematics \textbf{231:2} (2007), 279--291.
		
\bibitem{DT15}
O.~Dasbach, A.Tsvietkova, \textit{A refined upper bound for the hyperbolic volume of alternating links and the colored Jones polynomial}, Math Res. Letters, \textbf{22} (2015), 1047--1060.		
		
\bibitem{DT19}
O.~Dasbach, A.Tsvietkova, \textit{Simplicial volume of links from link diagrams}, Mathematical Proceedings of the Cambridge Philosophical Society, \textbf{166:1} (2019), 75--81.		
		
\bibitem{EV20-1} 
A.~Egorov, A.~Vesnin, \textit{Ideal right-angled polyhedra in Lobachevsky space},  Chebyshevskii Sbornik \textbf{21:2} (2020),  65--83. \url{http://doi.org/10.22405/2226-8383-2020-21-2-65-83.} 
		
\bibitem{EV20-2} 
A.~Egorov, A.~Vesnin, \textit{Volume estimates for right-angled hyperbolic polyhedra}, Rend. Istit. Mat. Univ. Trieste \textbf{52} (2020), 565--576.  Available at \url{https://rendiconti.dmi.units.it/volumi/52/029.pdf}


\bibitem{FKP} 
D.~Futer, E.~Kalfagianni, and J.\,S.~Purcell, \textit{Dehn filling, volume, and the Jones polynomial}, J. Differential Geom. \textbf{78:3} (2008), 429--464. 


\bibitem{GF06}
F.~Gu\'{e}ritaud, D.~Futer \textit{On canonical triangulations of once- punctured torus bundles and two-bridge link complements},   Geom. Topol. \textbf{10} (2006), 1239--1284.

\bibitem{Ka97}
R.\,M.~Kashaev, \textit{The hyperbolic volume of knots from quantum dilogarithm}, Letters in Mathematical Physics \textbf{39} (1997). 269--275.

\bibitem{Ka96}
A.~Kawauchi, \textit{A survey of knot theory}, Birkh\"auser, Basel, 1996, 423~pp.	
		
\bibitem{Ke89}
R.~Kellerhals, \textit{On the volume of hyperbolic polyhedra}, Mathematische Annalen \textbf{285} (1989) 541--569. 

\bibitem{Ke22}
R.~Kellerhals, \textit{A polyhedral approach to the arithmetic and geometry of hyperbolic chain link complements,} preprint available at  \url{https://homeweb.unifr.ch/kellerha/pub/Kellerhals.pdf}

\bibitem{Kw20}
A.~Kwon, \textit{Fully Augmented Links in the Thickened Torus}, preprint version available at \url{https://arxiv.org/abs/2007.12773}. 
		
\bibitem{LM}		
C.~Livingston,  A.\,H~ Moore, \textit{KnotInfo: Table of Knot Invariants}, available at \url{http://knotinfo.math.indiana.edu} 

	
\bibitem{La04}
M.~Lackenby, \textit{The volume of hyperbolic alternating link complements. With an appendix by I.~Agol and D.~Thurston},   
Proc. London Math. Soc. \textbf{88} (2004),  204---224.
		
\bibitem{MMT20} 
J.S.~Meyer, C.~Millichap, R.~Trapp, \textit{Arithmeticity and hidden symmetries of fully augmented pretzel link complements}, New York J. Math. \textbf{26} (2020), 149--183. 
		
\bibitem{Mi82}
J.~Milnor, \textit{Hyperbolic geometry: the first 150 years}, Bulletin Amer. Math. Soc. \textbf{6} (1982), 9--24.
		
\bibitem{MY05}
J.~Murakami, M.~Yano, \textit{On the volume of a hyperbolic and spherical tetrahedron}, Communication in Analysis and Geometry \textbf{13:2} (2005), 379-400.

\bibitem{PV09}
C.~Petronio, A.~Vesnin, \textit{Two-sided bounds for the complexity of cyclic branched coverings of two-bridge links}, Osaka J. Math. \textbf{46} (2009), 1077--1095.

\bibitem{Pu11}
J.\,S.~Purcell, \textit{An introduction to fully augmented links}, Interactions between hyperbolic geometry, quantum topology and number theory, Contemp. Math., \textbf{541}, Amer. Math. Soc., Providence, RI, (2011), 205--220.

\bibitem{Pu20}
J.\,S.~Purcell, \textit{Hyperbolic knot theory}, Graduate Studies in Mathematics, Amer. Math. Soc. \textbf{209} (2020).

\bibitem{RHD07}
R.\,K.\,W.~Roeder, J.\,H.~Hubbard, W.\,D.~Dunbar, \textit{Andreev's theorem on hyperbolic polyhedra}, Ann. Inst. Fourier, Grenoble, \textbf{57:3} (2007), 825--882. 

\bibitem{Ro76}
D.~Rolfsen, \textit{Knots and Links}, AMS Chelsea Publishing, 1976, 439~pp.

\bibitem{Snap}
SnapPy, a computer program available at \url{https://snappy.math.uic.edu}

\bibitem{Th80}
W.~Thurston, \textit{The Geometry and Topology of Three-Manifolds}, Lecture notes, Princeton, 1980, available at \url{http://www.msri.org/publications/books/gt3m/} 

\bibitem{Us06}
A.~Ushijima, \textit{A volume formula for generalized hyperbolic polyhedra}, in Non-Euclidean Geometries. Andras Prekopa and Emil Molnar (Ed.), Mathematics and Its Applications \textbf{581} (2006), 249--265. 

\bibitem{Vi93}
E.\,B.~Vinberg, \textit{Volumes of non-Euclidean polyhedra}, Russian Math. Surveys, \textbf{48:2} (1993), 15--45. 

		
	
\end{thebibliography}
\end{document}